\documentclass[a4paper,11pt]{article}

\usepackage{amsmath,amsfonts,amssymb,amsthm,cancel}
\usepackage[english]{babel}

\newcommand{\myspace}{\qquad\qquad\qquad}

\newcommand{\Dn}{\frac{\partial}{\partial \nu}}
\newcommand{\cA}{{\mathcal A}}

\newcommand{\cD}{{\mathcal D}}

\newcommand{\cL}{{\mathcal L}}

\newcommand{\cP}{{\mathcal P}}
\newcommand{\cR}{{\mathcal R}}

\newcommand{\cU}{{\mathcal U}}

\newtheorem{theorem}{Theorem}[section]
\newtheorem{lemma}[theorem]{Lemma}
\newtheorem{proposition}[theorem]{Proposition}
\newtheorem{remark}[theorem]{Remark}
\newtheorem{remarks}[theorem]{Remarks}

\newtheorem{corollary}[theorem]{Corollary}

\newtheorem{problem}[theorem]{Problem}
\numberwithin{equation}{section}

\begin{document}
\title{\bf Feedback control of the acoustic pressure in ultrasonic wave propagation}

%

\author{
Francesca Bucci \\
{\small Universit\`a degli Studi di Firenze}\\
{\small Firenze, ITALY}\\
{\small \texttt{francesca.bucci(at)unifi.it} }
\and 
Irena Lasiecka \\
{\normalsize University of Memphis}\\
{\small Memphis, TN, USA}\\
{\small IBS, Polish Academy of Sciences}\\
{\small \texttt{lasiecka(at)memphis.edu}} 
}

\date{}
\maketitle

\begin{abstract}
\noindent
Classical models for the propagation of ultrasound waves are the Westervelt equation, the
Kuznetsov and the Khokhlov-Zabolotskaya-Kuznetsov equations. 
The Jordan-Moore-Gibson-Thompson equation is a prominent example of a Partial Differential Equation
(PDE) model which describes the acoustic velocity potential in ultrasound wave propagation, where 
the paradox of infinite speed of propagation of thermal signals is eliminated; the use of the constitutive Cattaneo law for the heat flux, in place of the Fourier law, accounts for its being of third order in time.
Aiming at the understanding of the fully quasilinear PDE, a great deal of attention has
been recently devoted to its linearization -- referred to in the literature as the 
Moore-Gibson-Thompson equation -- whose mathematical analysis is also of independent 
interest, posing already several questions and challenges. 
\\
In this work we consider and solve a quadratic control problem associated with the linear
equation, formulated consistently with the goal of keeping the acoustic pressure close 
to a reference pressure during ultrasound excitation, as required in medical and industrial applications.
While optimal control problems with smooth controls have been considered in the recent
literature, 
we aim at relying on controls which are just $L^2$ in time;
this leads to a singular control problem and to non-standard Riccati equations. 
In spite of the unfavourable combination of the semigroup describing the free dynamics that is {\em not}
analytic, with the challenging pattern displayed by the dynamics subject to boundary control, a feedback
synthesis of the optimal control as well as well-posedness of operator Riccati equations are established.  
\end{abstract}

\medskip
\noindent
{\bf Key words}: ultrasound waves, optimal boundary control, absorbing boundary conditions, high intensity focused ultrasound, 
singular  control, nonstandard Riccati equations, feedback synthesis

\section{Introduction and motivation}
PDE models for the propagation of ultrasound waves -- more specifically, high intensity ultrasound propagation (HIUP) -- are relevant to a number of medical and industrial applications.
To name but a few, lithotripsy, thermoterapy, (ultrasound) welding, sonochemistry; 
cf., e.g., \cite{dreyer}. 
The excitation of induced acoustic fields in order to attain a given task, such as destroying certain `obstacles' (stones in kidneys or deposits resulting from chemical reactions), renders the presence of control functions within the model well-founded. 
 
The subject of the present investigation is an optimal control problem for a third order in time PDE,
referred to in the literature as the Moore-Gibson-Thompson equation, which is the linearization of 
the Jordan-Moore-Gibson-Thompson (JMGT) equation, arising in the modeling
of ultrasound waves; see \cite{jordan,jordan1}, \cite{kalten-eect}, \cite{straughan}.
In contrast with the renowned Westervelt (\cite{westervelt}) and Kuznetsov equations, the JMGT equation displays a {\em finite} speed of propagation of acoustic waves,
thereby providing a solution to the infinite speed of propagation paradox.
This is achieved by replacing the Fourier's law of heat conduction by the Cattaneo law
(\cite{cattaneo_1958}); the distinct constitutive law brings about an additional time
derivative of the acoustic velocity field (or acoustic pressure).

Restricting the analysis to the relevant spatial dimensions $n=2,3$, a Neumann boundary control will be acting as a force on a manifold $\Gamma_0$ of dimension $n-1$; $\Gamma_0$ will eventually represent a boundary portion of a bounded domain $\Omega \subset \mathbb{R}^n$.
(It is an established procedure to reduce the analysis of wave processes on {\em unbounded} domains to boundary or initial/boundary value problems (IBVP) on {\em bounded} domains via the introduction of artificial boundaries.)
Thus, {\em absorbing} boundary conditions (BC) will be taken on a complementary part of the boundary
$\Gamma_1 =\partial \Omega \setminus \Gamma_0$; see section~\ref{ss:setting}.
We shall assume that the two parts of the boundary do not intersect.
The optimal control problem arises from the minimization of the acoustic pressure in $\Omega$.
This setup, which is motivated by significant applications and technologies, has been already adopted
in the literature in connection with the said nonlinear PDE's; 
see \cite{kalten-eect}, \cite{clason-etal_2009}, \cite{vania}, \cite{vania1},\cite{clason-k} and references therein. 

From the mathematical point of view, two main challenges appear.
The first one is due to the presence of {\em boundary} controls, which naturally bring about
unbounded input operators $B$ into the (linear) abstract state equation $y'=Ay+Bg$; see \cite{bddm}, \cite{redbook}.  
It is well known that this issue can be dealt with by exploiting the additional regularity of the
PDE dynamics: this occurs in the case of parabolic-like dynamics, plainly governed by analytic semigroups $e^{At}$.
The reader is referred to the classical texts \cite{bddm} and \cite[Vol.~I]{redbook} for a thorough study of the Linear-Quadratic (LQ) problem for parabolic-like PDE's, along with the related differential
and algebraic Riccati equations.
\\
(We note that the same is actually valid in the case of PDE problems whose corresponding abstract control systems satisfy the so called {\em singular estimates} for $e^{At}B$, even if the semigroup 
$e^{At}$ is not analytic (\cite{las-cbms}).
And, further, appropriate regularity properties can be displayed by certain coupled systems of hyperbolic-parabolic PDE's subject to boundary control -- including thermoelastic systems, acoustic-structure and fluid-structure interactions --, which ensure the solvability of the associated optimal control problems (with quadratic functionals), along with well-posed Riccati equations.
The ultimate finite and infinite time horizon theories, as well as references to the motivating
PDE systems, are found respectively in \cite{abl-2005} and \cite{abl_2013}.)

Returning to the PDE under investigation, as we know from \cite{marchand} and \cite{kalten-las-mar_2011}, the dynamics of the (uncontrolled) SMGT equation, with classical Dirichlet or Neumann
BC, is described by a {\em group} of operators, displaying an intrinsic hyperbolic character, and hence a lack of regularity of its dynamics. 
In addition, a major challenge is brought about by the presence -- that cannot be eluded -- of the time derivative of the control function $g(t,x)$ within the control system, which becomes
\begin{equation} \label{e:g_t}
y'=Ay+B_0g+B_1g_t\,, 
\end{equation}
whereas on the other hand, penalization involves only the $L^2$ (in time) norm of the controls. 
This means that the cost functional is not coercive with respect to $g_t$. 
The resulting linear-quadratic problem becomes {\em singular}.
It must be recalled that these features have been already encountered and dealt with
in the study of optimal boundary control of (second-order in time) wave equations with structural damping; see the former study \cite{bucci_1992} and the subsequent analysis and solutions proposed in \cite{LLP} and \cite{LPT}.
Because of the strong damping, in the aforementioned case the free dynamics yields an analytic semigroup, along with an enhanced regularity of the control-to-state map; this feature has been exploited in the studies \cite{bucci_1992}, \cite{trig_1994}, \cite{LLP}, \cite{LPT}. 
Instead, the present PDE problem is of hyperbolic type.

The goal of the present paper is to provide a framework for such class of singular control problems, in the case of a hyperbolic-like dynamics which intrinsically does not exhibit 
regularizing effects on its evolution.
It is important to emphasize that while the singularity of the control is reflected in difficulties when treating {\em time} dependence, unbounded inputs affect the analysis of {\em space} dependence. 
So, the infinite-dimensional aspect of evolution is at the heart of the problem studied. 
To the authors' best knowledge this is a first investigation where a singular control 
problem and the control system \eqref{e:g_t} appear simultaneously, 
in an infinite dimensional context and with a general semigroup governing the free dynamics.  

\subsection{The nonlinear model and its linearization}
The Jordan-Moore-Gibson-Thompson (JMGT) equation is one of the fundamental equations in nonlinear acoustics
which describes wave propagation in viscous thermally relaxing fluids.
Its linearization is found in the literature as the Moore-Gibson-Thompson (MGT) equation.
(In recognition of the original work on it by Stokes (\cite{stokes}), it might rather be termed 
Stokes-Moore-Gibson-Thompson equation, as Pedro Jordan himself suggested; 
hence the acronym SMGT (in place of MGT) will be utilized throughout the paper.)
The fully nonlinear PDE, that is the JMGT equation, is the following one:
\begin{equation}\label{e:jmgt}
\tau \psi_{ttt} + \psi_{tt}-c^2\Delta \psi - b\Delta \psi_t=
\frac{\partial}{\partial t}\Big(\frac1{c^2}\frac{B}{2A}\psi^2_t+ |\nabla \psi|^2\Big)
\end{equation}
where $\tau > 0 $ is a time relaxation parameter, the unknown $\psi=\psi(t,x)$ is the {\em acoustic velocity potential}, the space variable $x$ varies in a bounded domain $\Omega\subset \mathbb{R}^n$, $c$ is the speed of sound,  parameter $b$ stands for diffusivity, $\alpha >0 $ is a damping parameter  and $A, B$ are suitable nonlinearity  constants; then, $-\nabla \psi$ is the acoustic {\em particle velocity}. 

When $\tau=0$ the model becomes the Kuznetsov equation, that is 
\begin{equation}\label{e:kuzn}
\psi_{tt}-c^2\Delta \psi - b\Delta \psi_t=
\frac{\partial}{\partial t}\Big(\frac1{c^2}\frac{B}{2A}\psi^2_t+ |\nabla \psi|^2\Big)\,,
\end{equation}  
a (second order in time) quasilinear PDE characterized by an infinite speed of propagation.
The positive diffusivity coefficient $b$ provides a regularizing effect on its evolution;
the corresponding linearized equation is of parabolic type, as its dynamics is governed by
an analytic semigroup. 
Instead, as found out in the former works \cite{kalten-eect} and \cite{marchand}, in the case
$\tau > 0$ the PDE turns into a finite speed of propagation and to a hyperbolic character.  
%

Optimal control problems with quadratic functional for both the Kuznetsov and Westervelt
equations have been studied first in \cite{clason-etal_2009} and \cite{clason-k}; 
see also \cite{kalten-eect}. 
The latter reads as 
\begin{equation*} 
u_{tt}-c^2\Delta u - b\Delta u_t= \beta \frac{\partial^2}{\partial t^2}\big(u^2\big)
\end{equation*}
in terms of the acoustic pressure $u$, where $\beta > 0$ is a suitable parameter of nonlinearity.
\\
(The relation $u=\rho \psi_t$ between the acoustic pressure and velocity potential 
-- $\rho(x)$ being the mass density -- allows another formulation of the Kuznetsov equation, 
with the pressure as the unknown variable.)  
Then, the ultrasound excitation on a certain manifold $\Gamma_0$ (of dimension~$n-1$)
can be represented by means of the Neumann boundary condition 
$\frac{\partial u}{\partial \nu}=g$ on $\Gamma_0$, where $g$ is the control function.
A question which arises is to minimize appropriate cost functionals associated with the controlled PDE.   

In the works \cite{clason-etal_2009} and \cite{clason-k} quadratic functionals of tracking type are taken into consideration, such as 
\begin{equation*}
J(g)= \frac{1}{2}\int_{\Omega} |u(T,x)-u^d(x)|^2\, dx 
+ \frac{\alpha}{2}\int_0^T\!\!\!\int_{\Gamma_0} |g|^2 \,d \sigma\,dt
\end{equation*}
and
\begin{equation*}
J(g)= \frac{1}{2}\int_0^T \!\!\!\int_{\Omega} |u - u^d|^2\, dx\,dt 
+ \frac{\alpha}{2}\int_0^T\!\!\!\int_{\Gamma_0} |g|^2 \,d \sigma\,dt\,,
\end{equation*}
respectively, where $u^d$ is a given reference pressure;
the class of admissible controls $G^{ad}$ is a suitably chosen space 
whose topology is induced by 
\begin{equation} \label{e:smooth-space}
H^1(0,T; H^{1/2}(\Gamma_0)) \cap H^2(0,T;H^{-1/2}(\Gamma_0))\,.
\end{equation}
A critical role in these studies was played by (i) the assumption that $G^{ad}$ represents a space of smooth controls -- more precisely, differentiable in time and subject to appropriate compatibility conditions (with respect to initial data) --, as well as 
(ii) the control constructed is an open-loop one, rather than a feedback one;
(iii) the solutions considered are suitably small and the state equation is of parabolic type.
\\ 
For such class of controls existence, uniqueness of solutions for {\it  small data}  (due to quasilinearity) has
been derived (\cite{kalten-las-pos_2012}, \cite{kalten-eect}). 
The optimal control is characterized via the Pontryagin Maximum Principle; see \cite{clason-etal_2009}. 

The present study, although focused on a simpler linear equation, 
departs from the avenues (i)--(iii), guided by two major goals. 
On one hand, we aim at minimizing a quadratic functional that penalizes controls functions in the $L^2$
(in time and space) norm, with (state) solutions under consideration not necessarily smooth (in space). 
A set of admissible controls that possess a low regularity is consistent with physical and engineering applications; see, e.g., \cite{dreyer}. 
In addition, feedback or closed-loop controls 
are of particular interest. 

On the other hand, as already apparent in the case of the Westervelt equation -- as well
as in the case of its linearization, that is the strongly damped wave equation (\cite{bucci_1992}) --, the
modeling of boundary control actions naturally brings about the time derivative of the control function,
which is somehow ``hidden'' within the PDE problem. 
This intrinsic analytical aspect will be made clear later -- once we derive the 
input-to-state solution formula.
If one were to pursue such a study in the case of the JMGT equation, a natural choice would
be to begin with the linear dynamics: it is already there where non-smoothness of controls will provide sufficient challenge.
In fact, the minimization problem overall $L^2$ controls may not ensure an optimal solution even in the linear case, as already noted in \cite{LPT}.
We shall confirm this finding in the case of the problem under consideration.

The above suggests that appropriate adjustments in the formulation of the problem and its modeling need to be made.
We shall show that by enlarging slightly the class of controls resolves the issue of existence of optimal solution.
Having  established this, we shall proceed with the optimality analysis and the construction of a
feedback control for the PDE which will still display `rough' states. 
However, the feedback solution will be shown to generate sufficiently regular outputs which can be used to control the system on-line -- via the solution to a {\em non-standard} differential Riccati Equation (RE).
The well-posedness of these corresponding non-standard Riccati equations provides a contribution of independent interest. 
In fact, the construction of solutions to the RE requires the extension of the dynamics to extrapolation spaces with very low regularity. 
This is needed in order to make the dynamics invariant.  

\smallskip
To recapitulate: 
the novel contribution of the present work pertains to optimal feedback control of the acoustic 
SMGT equation; the closed-loop control will be generated by an appropriate non-standard Riccati equation.
(The non-standard structure is due to the singular nature of the optimization problem.)
Focus is placed on the linearized version of the model, which already provides significant challenges in terms of the underlying analysis and constitutes a necessary step for a further treatment of nonlinear problems.
The expectation 
is that once a solution is given for the optimal feedback control of the linearized dynamics, such control may be used 
for the nonlinear problem, which then will have to be considered with small initial data.
A similar approach has been pursued successfully in the case of the Navier-Stokes equations;
cf. \cite{barbu}, \cite{barbu1}. 


\subsection{Mathematical setting} \label{ss:setting}
We consider the problem of controlling the acoustic excitation on a certain closed region 
$\Gamma_0$ while maintaining the acoustic pressure below a certain threshold; $\Gamma_0$ will be
subsequently identified as a part of the boundary of an introduced bounded domain $\Omega$. 
Then, as usually done in the study of wave propagation phenomena in unbounded spatial domains, an artificial boundary $\Gamma_1$ is introduced in order to limit the area of observation/computation.
The {\em absorbing} boundary conditions (BC) on $\Gamma_1$ are then used to avoid reflections:
roughly, no waves can `come back'. 
Accordingly, and consistently with the analysis carried out in \cite{clason-etal_2009} 
(on a classical nonlinear model for ultrasound wave propagation like the Westervelt equation), 
we will complement the SMGT equation with the BC which are the most pertinent: namely, 
\begin{itemize}
\item
Neumann {\em boundary control} acting on 
$\Gamma_0$ (the so called excited boundary); $g$ below represents a surface force;
\item
absorbing BC on the complement $\Gamma_1=\partial\Omega\setminus \Gamma_0$ 
(the so called {\em absorbing boundary}).
\end{itemize}
Thus, the boundary value problem (BVP) is as follows:
\begin{equation} \label{e:bvp-0}
\begin{cases}
\tau u_{ttt} + \alpha u_{tt} -c^2 \Delta u -b \Delta u_t =0 & \textrm{on $(0,T)\times\Omega$}
\\
\frac{\partial u}{\partial \nu}=g & \textrm{on $(0,T)\times\Gamma_0$}
\\
\frac{\partial u}{\partial \nu}+\frac{1}{c}u_t=0 & \textrm{on $(0,T)\times\Gamma_1$}
\end{cases}
\end{equation}
to be supplemented with initial conditions.

\noindent
Aiming at studying optimal control problems with quadratic functionals associated with the IBVP \eqref{e:bvp-0}, the following features need to be taken into account:  
\\
(i) {\em finite} time horizon problems, in the absence of penalization of the final time are the most pertinent ones (e.g., in lithotripsy);
\\
(ii) with $u$ representing the acoustic pressure, the quantity to be minimized (under the action of the surface force $g$) is 
$\|u-u^d\|_{L^2(0,T;L^2(\Omega))}^2$, where $u^d$ is a reference pressure;
\\
(iii)
longer times (i.e. $T=+\infty$) might be taken into consideration (e.g., in connection with thermotherapy).

Depending on the applications, different cost functionals may be considered.
In what follows we shall focus on the physically significant minimization of the following 
cost functional (of tracking type): 
\begin{equation} \label{e:cost-funct}
J(g) = \int_0^T\!\!\!\int_{\Omega} |u-u^d|^2 \,dx\, dt + \int_0^T\!\!\! \int_{\Gamma_0} |g|^2 \,d\sigma\, dt\,.
\end{equation}

\begin{remark}
\begin{rm}
The fact that the functional cost penalizes the control $g$ only in the $L^2$ norm renders the optimization problem a singular one. 
Indeed, if one penalizes also the velocity $g_t$ of the control, then we would obtain a standard boundary control problem with coercive cost functional.
\end{rm} 
\end{remark}

Control problems associated with acoustic equations (Westervelt, Kuznetsov, JMGT ones) have
been recently studied in the literature; see the review paper \cite{kalten-eect}. 
However, the principal difference is that the present minimization involves control functions which belong to $L^2(\Sigma)$, $\Sigma := (0,T)\times \Gamma_0$, 
rather than more regular -- time-space differentiable -- controls (see \eqref{e:smooth-space}, 
and the optimal control problems studied in \cite{clason-etal_2009} and \cite{vania}). 
In addition, control laws provided in the past literature were {\it open loop} controls. 
Our goal is to construct {\it feedback control} with controls of limited regularity and control gains represented by solutions to Riccati equation. 
This last aspect is the main trait of our contribution. 
A brief outline-guide to the paper follows below. 

\smallskip
In order to state our results and to explain ramifications of the low regularity of the control, 
it is necessary to derive an abstract input-to-state formula of the IBVP problem, within the realm of classical control theory.
This means we will seek an explicit representation for the map 
\begin{equation}\label{gu}
g \longrightarrow (u,u_t,u_{tt})
\end{equation}
This will be accomplished in the next Section~\ref{s:semigroup-perspective} by using semigroup theory. 
Starting with uncontrolled dynamics and its representation via generator of a strongly continuous semigroup, we shall then proceed introducing boundary controls into the ``variation of parameters formula'' which will provide an explicit map \eqref{gu} -- singular and defined on appropriately selected extrapolation spaces, though. 

In the next step we shall formulate control problems associated with the input-state dynamics and we shall discuss existence and non-existence of optimal solutions. 
The final result pertaining to well-posedness of Riccati equations and to the feedback synthesis 
of the optimal control is presented in Section~\ref{s:auxiliary-and-riccati}.  
It is important to notice here that in spite of the singularity of input-state dynamics, the feedback synthesis and the resulting Riccati equations are defined and well-posed on the basic state and control spaces. 
This is due to the effects of the observation. 

The proofs of the auxiliary and main results are deferred to Sections~\ref{s:proofs_1} 
and \ref{s:proofs_2}. 
The proofs will rely on techniques introduced in the study of the LQ problem for hyperbolic-like equations with unbounded inputs, where the dynamics does not provide beneficial regularizing effects. 
To handle this issue, we  establish appropriate bounds by exploiting structural properties of the observation; see \cite[Vol.~II]{redbook}.


\section{Input-to-state formulation of the PDE problem} \label{s:semigroup-perspective}
A prerequisite step for the understanding of the control-theoretic properties of the 
initial/boundary value problem (IBVP)
\begin{equation} \label{e:ibvp-1}
\begin{cases}
\tau u_{ttt} + \alpha u_{tt} -c^2 \Delta u -b \Delta u_t =0 & \textrm{on $(0,T)\times\Omega$}
\\
\frac{\partial u}{\partial \nu}=g & \textrm{on $(0,T)\times\Gamma_0$}
\\
\frac{\partial u}{\partial \nu}+\frac{1}{c} u_t= 0& \textrm{on $(0,T)\times\Gamma_1$}
\\
u(0,x)=u_0(x)\,,\; u_t(0,x)=u_1(x)\,; u_{tt}(0,x)=u_2(x) & \textrm{on $\Omega$}
\end{cases}
\end{equation}
for the SMGT equation is to introduce the corresponding abstract operator model in 
an appropriate function spaces.

\subsection{Abstract setup. Preliminary analysis}
In order to incorporate into the equation the boundary control action on $\Gamma_0$,
along with the absorbing BC on $\Gamma_1$, we follow a well-established method. 

Let $\cA$ be the realization of $-\Delta$ in $L^2(\Omega)$ with Neumann BC: namely,
\begin{equation*}
\cA=-\Delta\,, \quad 
\cD(\cA) =\Big\{ f\in H^2(\Omega): \; \frac{\partial f}{\partial \nu}\Big|_{\partial \Omega}
=0\Big\}\,.
\end{equation*}
It is well known that $\cA$ is not boundedly invertible on $L^2(\Omega)$; it has
bounded inverse on
\begin{equation*}
L^2_0(\Omega):=L^2(\Omega)/\ker(\cA)
=\Big\{f\in L^2(\Omega)\colon \int_\Omega f\,d\Omega=0\Big\}\,,
\end{equation*}
where $\ker(\cA)$ is the null space of $\cA$ spanned by the normalized constant functions.
Then, introduce the Green maps $N_i$, $i=0,1$, which define appropriate harmonic extensions into $\Omega$ of data defined on $\partial\Omega$. 
More precisely, for $\varphi\in L^2(\Gamma_0)$, $N_i$ will be defined as follows: 
\begin{equation}\label{e:N_i}
N_i\colon\varphi\longmapsto N_i\varphi=:v \;
\Longleftrightarrow \;
\begin{cases}
\Delta v -v =0 & \textrm{on $\Omega$}
\\[1mm]
\frac{\partial v}{\partial \nu}=\varphi & \textrm{on $\Gamma_i$}
\\[1mm]
\frac{\partial v}{\partial \nu}=0 & \textrm{on $ \partial \Omega  \setminus  \Gamma_i$.}
\end{cases}
\end{equation}
Either elliptic problem that defines the operator $N_i$ in \eqref{e:N_i} admits a unique
solution $v_i\in H^{3/2}(\Omega)$, for (respective) boundary data $\varphi\in L^2(\Gamma_i)$,
$i=0,1$.
Then, by elliptic theory one has for each $i=0,1$ and any positive $\sigma<3/4$ 
\begin{equation} \label{e:neumann-regularity}
N_i \; \textrm{continuous}\colon L^2(\Gamma_i) \longrightarrow H^{3/2}(\Omega)
\subset H^{3/2-2\sigma}(\Omega)\equiv \cD((I+\cA)^{3/4-\sigma})\,,
\end{equation}
with identification of the Sobolev spaces $H^s(\Omega)$ with the fractional powers of the operator $(I+\cA)$, and equivalent norms, that will be especially useful in the sequel. 

If now $N_i^*$ denote the respective adjoint operators of $N_i$, $i=0,1$ -- defined by 
$(N_i\phi,w)_{L^2(\Omega)}=(\phi,N_i^*w)_{L^2(\Gamma_i)}$ --, it then follows
for each $i=0,1$ and any $\sigma\in (0,3/4)$,
\begin{equation*}
(I+\cA)^{3/4-\sigma}N_i\in \cL(L^2(\Gamma_i),L^2(\Omega))\,,
N^*_i(I+\cA)^{3/4-\sigma}\in \cL(L^2(\Omega),L^2(\Gamma_i))\,.
\end{equation*} 
As in \cite{trig_1994}, a computation which utilizes the (second) Green Theorem yields, for $f\in \cD(\cA)$, 
the following fundamental trace results:
\begin{equation} \label{e:trace-result-ni}
N_i^*(\cA+I) f  =f|_{\Gamma_i} \qquad \textrm{$i=0,1$.}
\end{equation}
(For the reader's convenience: take $v\in \cD(\cA)$, $\varphi\in L^2(\Gamma_0)$,
and compute
\begin{equation*}
\begin{split}
& -\big(N_0^*(\cA+I) v,\varphi\big)_{\Gamma_0}= \big(-(\cA+I) v,N_0\varphi\big)_\Omega
= (\Delta v,N_0\varphi)_\Omega-(v,N_0\varphi)_\Omega=
\\
& \qquad = \big(v,\Delta (N_0\varphi)\big)_\Omega
+ \cancel{\Big(\frac{\partial v}{\partial\nu},N_0\varphi\Big)_{\partial\Omega}}
- \Big(v,\frac{\partial N_0\varphi}{\partial\nu}\Big)_{\partial\Omega}
-(v,N_0\varphi)_\Omega=
\\
& \qquad = (v,N_0\varphi)_\Omega
-(v,\varphi)_{\Gamma_0}-(v,N_0\varphi)_{\Omega}
=-(v,\varphi)_{\Gamma_0}\,.
\end{split}
\end{equation*}
The above shows that \eqref{e:trace-result-ni} holds true when $i=0$; the case $i=1$ is proved in the same way.
We note that it has been used that since $v$ belongs to $\cD(\cA)$,
then $\frac{\partial v}{\partial\nu}=0$ on $\partial\Omega$; in addition, the definition
of $N_0\varphi$ in \eqref{e:N_i} -- as the solution of an elliptic problem -- gives
in particular $\Delta (N_0\varphi)=N_0\varphi$.)

\smallskip
In view of the definition of the introduced operators $N_i$, $i=0,1$, we see that
\begin{equation*}
\begin{cases}
(\Delta -I)\big(u+\frac1{c}N_1 u_t|_{\Gamma_1}-N_0g\big)=(\Delta -I)u  &  
\textrm{on $\Omega\times (0,T)$}
\\[1mm]
\frac{\partial}{\partial \nu}(u+\frac1{c}N_1 u_t-N_0g)=0 
& \textrm{on $\Gamma_0\times (0,T)$}
\\[1mm]
\frac{\partial}{\partial \nu}(u+\frac1{c}N_1 u_t-N_0g)=0\,; 
& \textrm{on $\Gamma_1\times (0,T)$}
\end{cases}
\end{equation*}
proceeding formally we get 
\begin{equation*}
\begin{split}
\Delta u &= (\Delta -I)\big(u+\frac1{c_1}N_1 u_t|_{\Gamma_1}-N_0g\big) +u\,,
\\
\Delta u_t &= (\Delta -I)\big(u_t+\frac1{c_1}N_1 u_{tt}|_{\Gamma_1}-N_0 g_t\big) +u_t\,, 
\end{split}
\end{equation*}
which enable us to rewrite the SMGT equation as 
\begin{equation*}
\begin{split}
& \tau u_{ttt}+\alpha u_{tt} -c^2 (\Delta - I)\big(u+\frac1{c}N_1 u_t|_{\Gamma_1}-N_0g\big)-c^2 u-
\\
& \myspace - b(\Delta -I)\big(u_t+\frac1{c}N_1 u_{tt}|_{\Gamma_1}-N_0 g_t\big) - b u_t=0\,,
\end{split}
\end{equation*}
where $\frac{\partial}{\partial \nu}\big(u+\frac1{c_1}N_1 u_t|_{\Gamma_1}-N_0g\big)\big|_\Gamma=0$.
Thus, by using the trace results \eqref{e:trace-result-ni}, the BVP \eqref{e:bvp-0} for the SMGT equation
translates to the following abstract equation, where both the absorbing BC on $\Gamma_1$ and the boundary control action on $\Gamma_0$ are incorporated:
\begin{equation*}
\begin{split}
& \tau u_{ttt}+\alpha u_{tt} +c^2 (\cA+I)\Big[u+\frac1{c}N_1 N_1^*(\cA+I)u_t-N_0 g\Big]
-c^2 u+
\\[1mm]
& \myspace + b(\cA+I)\Big[u_t+\frac1{c}N_1N_1^*(\cA+I) u_{tt}-N_0 g_t\Big] - b u_t=0\,,
\end{split}
\end{equation*}
that is 
\begin{equation}\label{e:controlled-eq}
\begin{split}
& \tau u_{ttt}+\alpha u_{tt} +c^2 \cA u +c(\cA+I)N_1 N_1^*(\cA+I)u_t + b \cA u_t +
\\[1mm]
& \myspace +\frac{b}{c} (\cA+I)N_1 N_1^*(\cA+I) u_{tt} = c^2 (\cA+I)N_0 g + b (\cA+I)N_0 g_t\,;
\end{split}
\end{equation} 
the equality is understood with respect to the duality pairing, i.e. in $[\cD(\cA)]'$. 

The third order abstract equation \eqref{e:controlled-eq} gives rise readily to a first order control system,
initially defined on an extended space $L^2(\Omega) \times L^2(\Omega) \times [\cD(\cA)]'$:
\begin{equation}\label{e:1order-system}
\frac{d}{dt}\begin{pmatrix}
u\\
u_t\\
u_{tt}
\end{pmatrix} = A  \begin{pmatrix}
u\\
u_t\\
u_{tt}
\end{pmatrix} + B_0 g + B_1 g_t\,,
\end{equation}
where the operator describing the {\em free} dynamics is
{\small
\begin{equation}\label{e:generator}
A= 
\begin{pmatrix}
0 & I & 0
\\[1mm]
0 & 0 & I
\\[1mm]
- \tau^{-1} c^2\cA & -\tau^{-1} \big[b \cA+c(\cA+I)N_1N_1^*(\cA+I)\big] 
& - \tau^{-1} \big[\alpha I +\frac{b}{c}(\cA+I)N_1 N_1^*(\cA+I)\big]
\end{pmatrix}
\end{equation}
}
while the input operators $B_i\in \cL(L^2(\Gamma_0),[\cD(\cA)]'), i=0,1$, are given by
\begin{equation}\label{e:input-operators}
B_0=
\begin{pmatrix}
0 
\\[1mm]
0 
\\[1mm]
\tau^{-1}c^2 (\cA+I)N_0
\end{pmatrix}\,,
\qquad 
B_1=
\begin{pmatrix}
0 
\\[1mm]
0 
\\[1mm]
\tau^{-1}b (\cA+I)N_0
\end{pmatrix} 
=\frac{b}{c^2}B_0\,.
\end{equation}
The (free dynamics) operator $A$ in \eqref{e:generator} will be shown to generate a $C_0$-semigroup on the space $Y = H^1(\Omega) \times H^1(\Omega) \times L^2(\Omega)$. 

\begin{remark}
\begin{rm}
The first order equation \eqref{e:1order-system} is a control system in (extended to) the dual space
$[\cD(A^*)]'$;
this is due to the fact that $A^{-1} B_i \in \cL(U,Y)$, $i=0,1$, as it will be verified later.
However, the given formulation involves the time derivative of control, which is not supported by the
cost functional; as a consequence, the minimization problem lacks coercivity.
To cope with this, we will follow \cite{LLP}: 
integration by parts in the input-to-state formula enables to `eliminate' the time derivative of the control function, however with the drawback that the states will become `rougher'.
The smoothing properties of the observation operator $R$ -- here, {\em intrinsic} -- will play a major role
in the entire subsequent analysis, which will eventually bring about the solution of the optimization problem.
\end{rm}
\end{remark}
 
Before we proceed, let us consider the uncontrolled equation first. 
This step is necessary in order to formulate a correct notion of duality -- which is always with respect to
the generator of the semigroup underlying the dynamics. 


\subsection{The uncontrolled equation. Semigroup well-posedness}
In order to pinpoint the control-theoretic properties of the abstract system  
\eqref{e:1order-system} -- an ineludible preliminary step for the analysis 
of the optimal control problem --, we consider first the uncontrolled equation, that is
equation \eqref{e:controlled-eq} in the absence of the boundary action $g$.
With $g\equiv 0$, the equation \eqref{e:controlled-eq} reads as
\begin{equation}\label{e:free-eq}
\begin{split}
& \tau u_{ttt}+\alpha u_{tt} +c^2 \cA u +c(\cA+I)N_1 N_1^*(\cA+I)u_t + b \cA u_t +
\\
& \myspace + \frac{b}{c} (\cA+I)N_1 N_1^*(\cA+I) u_{tt}=0\,.
\end{split}
\end{equation} 
We follow an idea introduced 
and utilized in \cite{kalten-las-mar_2011} and \cite{marchand}.
Calculations below might appear formal: however, they are fully justified with respect to the duality in
$[\cD(A^*)]'$.
After having set $\tau=1$ for the sake of simplicity, the rewriting of equation \eqref{e:free-eq} as
\begin{equation}\label{e:free-eq-1}
(u_t+\alpha u)_{tt}+b \cA\Big(u_t+\frac{c^2}{b} u\Big) 
+\frac{b}{c}(\cA+I)N_1 N_1^*(\cA+I)\Big(u_{tt}+ \frac{c^2}{b} u_t\Big)=0\,,
\end{equation} 
suggests the introduction of the auxiliary variable 
\begin{equation} \label{e:def-of-z}
z:= u_t+\frac{c^2}{b} u\,.
\end{equation} 
The new variable $z$ plays a major role in deriving well-posedness
results for the third order equation \eqref{e:free-eq} in the unknown $u$;
this is because it allows to connect the (free) equation under investigation with
the following system in the unknowns $(u,z)$:
\begin{equation}\label{e:auxiliary-free-sys}
\begin{cases}
u_t=-\frac{c^2}{b} u+z
\\[1mm]
z_{tt}=-b \cA z -\frac{b}{c}(\cA+I)N_1 N_1^*(\cA+I)z_t -\gamma z_t+\gamma \frac{c^2}b z 
- \gamma \Big(\frac{c^2}b\Big)^2 u
\end{cases}
\end{equation}
where $\gamma:= \alpha - \frac{c^2}{b}$ will be assumed to be positive.
The explicit statement and proof of this claim, that is an immediate generalization of 
what done in \cite{kalten-las-mar_2011}, is given below for the reader's convenience
and the sake of completeness.
  
\begin{lemma}
The uncontrolled third order (in time) equation \eqref{e:free-eq} is equivalent to 
the coupled ODE-PDE system \eqref{e:auxiliary-free-sys}, with 
$\gamma=\alpha - \frac{c^2}{b}$.
\end{lemma}

\begin{proof}
The starting point is equation \eqref{e:free-eq-1} that is nothing but a rewriting of
\eqref{e:free-eq}.
With the new variable $z= u_t+c^2/b u$, the term $u_t+\alpha u$ in \eqref{e:free-eq-1}
is rewritten in terms of $z$ and $u$ as follows:
\begin{equation*}
u_t+\alpha u= z+\gamma u\,, \quad  \gamma:= \alpha - \frac{c^2}{b}\,,
\end{equation*}  
so that \eqref{e:free-eq-1} becomes
\begin{equation}\label{e:free-eq-2}
z_{tt}+b \cA z +\frac{b}{c}(\cA+I)N_1 N_1^*(\cA+I)z_t + \gamma u_{tt}=0\,.
\end{equation} 
On the other hand, using once again the definition of $z$ 
we see that $u_t=z-c^2/b \,u$, which gives 
\begin{equation}\label{e:u_tt}
u_{tt}=z_t-\frac{c^2}{b} u_t=z_t-\frac{c^2}{b} z+\Big(\frac{c^2}{b}\Big)^2 u\,;
\end{equation} 
the above, inserted in \eqref{e:free-eq-2} yields the following equation:
\begin{equation*}
z_{tt}+b \cA z +\frac{b}{c}(\cA+I)N_1 N_1^*(\cA+I)z_t + \gamma z_t-\gamma \frac{c^2}b z + 
\Big(\frac{c^2}b\Big)^2 u=0\,.
\end{equation*}

The latter second order in time equation for $z$, combined with
\eqref{e:u_tt} 
leads to the following coupled system of (second-order in time) equations in the 
unknowns $(u,z)$
\begin{equation*}
\begin{cases}
u_{tt}=z_t-\frac{c^2}{b} u_t
\\[1mm]
z_{tt}+b \cA z +\frac{b}{c}(\cA+I)N_1 N_1^*(\cA+I)z_t + \gamma z_t-\gamma \frac{c^2}b u_t =0\,;
\end{cases}
\end{equation*}
or, equivalently, to the coupled ODE-PDE system \eqref{e:auxiliary-free-sys}.
\end{proof}

We establish  semigroup well-posedness of the Cauchy problems associated with
system \eqref{e:auxiliary-free-sys} in three different function spaces.

\begin{theorem}[Equivalent system. Well-posedness, I] \label{t:first}
The (first order in time) system in the unknown $(u,z,z_t)$ corresponding to
system \eqref{e:auxiliary-free-sys} is well-posed in the space
\begin{equation*}
Y= \underbrace{H^1(\Omega)}_{u}\times \underbrace{H^1(\Omega)\times L^2(\Omega)}_{(z,z_t)}\,.
\end{equation*}
Its dynamics is described by a closed operator $\tilde{A}: \cD(\tilde{A})\subset Y\to Y$ which is the generator of a $C_0$-semigroup $e^{\tilde{A} t}$ on $Y$, $t\ge 0$.  
\end{theorem}

\begin{proof}
The second-order system \eqref{e:auxiliary-free-sys} is rewritten as 
a first-order system
\begin{equation*} 
\begin{pmatrix}
u\\
z\\
z_t
\end{pmatrix}_t= \tilde{A}\begin{pmatrix}
u\\
z\\
z_t
\end{pmatrix}\,,
\end{equation*}
with dynamics operator
\begin{equation*}
\tilde{A}=
\begin{pmatrix}
-\frac{c^2}{b}\,I & I & 0
\\[1mm]
0 & 0 & I
\\[1mm]
- \gamma\big(\frac{c^2}b\big)^2 \,I & -b \cA+\gamma \frac{c^2}b\,I & -\gamma I 
- \frac{b}{c}(\cA+I)N_1 N_1^*(\cA+I) 
\end{pmatrix}\,.
\end{equation*}
It is then natural to observe that the decomposition 
\begin{equation*}
\tilde{A}= \tilde{A}_1+ C_1+K_1
\end{equation*}
holds true, where we set
\begin{equation*}
\tilde{A}_1=
\begin{pmatrix}
-\frac{c^2}{b}\,I & 0 & 0
\\
0 & 0 & I
\\
0 & -b (\cA+I) & -\gamma I - \frac{b}{c}(\cA+I)N_1 N_1^*(\cA+I) 
\end{pmatrix}\,,
\end{equation*}
\begin{equation*}
C_1=\begin{pmatrix}
0 & I & 0
\\
0 & 0 & 0 
\\
-\gamma \big(\frac{c^2}{b}\big)^2\,I & 0 & 0 
\end{pmatrix}\,,
\qquad 
K_1=\begin{pmatrix}
0 & 0 & 0
\\
0 & 0 & 0
\\
0 & (\gamma \frac{c^2}{b}+b)\,I  & 0
\end{pmatrix}\,.
\end{equation*}
It is enough to single out the following respective features:  
\\
(i) the operator $\tilde{A}_1: \cD(\tilde{A}_1)\subset Y\longrightarrow Y$ is a 
(maximally) dissipative operator on 
\begin{equation*}
\underbrace{H^1(\Omega)}_{u}\times \underbrace{\cD((\cA+I)^{1/2})\times L^2(\Omega)}_{(z,z_t)}
\end{equation*}
and hence it is the generator of a $C_0$-semigroup of {\em contractions} 
$e^{\tilde{A}_1 t}$ on $Y$ (which, however, is {\em not} analytic);  
\\
(ii)
$C_1$ is a {\em bounded} operator from $Y$ into itself;
\\
(iii)
$K_1$ is a {\em compact} operator: in fact, with $f\in \cD((\cA+I)^{1/2})$ one has 
\begin{equation*}
\gamma \Big(\frac{c^2}{b}+b\Big)\,f
=\gamma \Big(\frac{c^2}{b}+b\Big)(\cA+I)^{-1/2}[(\cA+I)^{1/2}f]\,.
\end{equation*}
The generation of a $C_0$-semigroup $e^{\tilde{A} t}$ on $Y$ follows
by semigroup theory.
\end{proof}


\begin{remark}
\begin{rm}
The space $Y$ will provide an appropriate functional setting where the original uncontrolled
system is well-posed, and a state space for the optimal control problem under investigation.
It is however important to add that well-posedness remains valid in distinct functional spaces;
the corresponding results are stated below for the sake of completeness, while the relative proofs are
omitted. 
\end{rm}
\end{remark}


\begin{corollary}[Equivalent system. Well-posedness, II]
The uncontrolled problem is well-posed in
\begin{equation*}
Y_2= \underbrace{H^2(\Omega)}_{u}\times \underbrace{H^1(\Omega)\times L^2(\Omega)}_{(z,z_t)}\,.
\end{equation*}
\end{corollary}

Thus, in view of the definition of the domain of the generator $\tilde{A}$, that is
\begin{equation*}
\begin{split}
\cD(\tilde{A}) &= \big\{ (u,z,z_t)\in [H^1(\Omega)]^3\colon
z+\frac{1}{c}N_1\,N_1^*(\cA+I)z_t\in \cD(\cA)\big\}=
\\[1mm]
& = \Big\{ (u,z,z_t)\in H^1(\Omega)\times H^2(\Omega)\times H^1(\Omega)\,\colon
\; \frac{\partial z}{\partial \nu}\Big|_{\Gamma_0} =0 \,, \;
\Big[c\frac{\partial z}{\partial \nu}+z_t\Big]_{\Gamma_1} =0 \Big\}\,,
\end{split}
\end{equation*}
taking the dual $[\cD(\tilde{A})]'$ (duality with respect to $Y_2$), we are able to infer the following result.
\begin{corollary}[Equivalent system. Well-posedness, III]
The uncontrolled problem is well-posed in
\begin{equation}
Y_0 \sim \underbrace{H^1(\Omega)}_{u}\times \underbrace{L^2(\Omega)\times [H^1(\Omega)]'}_{(z,z_t)}\,,
\end{equation}
where $\sim$ indicates topological equivalence.
\end{corollary}

The next Theorem~\ref{t:first} summarizes    relevant  wellposedness  results which will be used throughout . 

\begin{theorem}[The uncontrolled equation. Well-posedness and stability] \label{t:wellposed_1}
With reference to the third order abstract equation \eqref{e:free-eq} describing the free dynamics,
the following statements hold true.
\begin{enumerate}
\item[i)]
The boundary value problem \eqref{e:bvp-0} with $g\equiv 0$ admits the abstract formulation
\eqref{e:free-eq} as a third order equation; 
equivalently, it is rewritten as a first order abstract system $y'=Ay$, where $y$ denotes the
state variable $(u,u_t,u_{tt})$.

\item[ii)]
The operator $A$ which governs the free dynamics, detailed in \eqref{e:generator},
is the generator of a $C_0$-semigroup $\{e^{At}\}_{t\ge 0}$ on the function space
$Y=H^1(\Omega)\times H^1(\Omega)\times L^2(\Omega)$.

\item[iii)] 
The semigroup $e^{At}$ is exponentially stable when $\gamma > 0$. 

\end{enumerate}
\end{theorem}

\begin{remark}
\begin{rm}
In  the critical case, when $\gamma =0$, it is expected that with $\Gamma_0$ subject to 
the ``star-shaped'' Geometric Condition (cf.~\cite{lagnese}) the resulting semigroup is exponentially stable.
\end{rm}  
\end{remark}

\begin{remarks} 
\begin{rm}
The first assertion in Theorem \ref{t:wellposed_1} establishes the existence of a linear semigroup defined
on $Y$ which describes the original uncontrolled dynamics. 
It is worth noting that if the SMGT equation is complemented with either Dirichlet or Neumann
BC the same result holds true, as it was first proved in \cite{marchand} and \cite{kalten-las-mar_2011};
in that case the semigroup is actually a {\em group} on $Y$. 
Instead, the group property is not valid any more in the presence of absorbing BC on $\Gamma_1$. 

The studies \cite{kalten-las-mar_2011} and \cite{marchand} -- the latter, providing a clarifying 
spectral analysis -- obtain that (still in the case of Dirichlet or Neumann BC)
the semigroup $e^{tA}$ is exponentially stable on the factor space $Y/{\ker(A)}$, 
provided $\gamma > 0$; it is marginally stable when $\gamma =0$ and unstable when $\gamma < 0$.
In the present case, assuming appropriate geometric conditions on $\Gamma_0$, the absorbing boundary conditions
turn marginal stability ($\gamma =0$) to stability. 
This issue has not been fully investigated so far, yet it is expected that the multipliers' method combined
with a background on wave equations would provide the tools. 
\end{rm}
\end{remarks}

\begin{remark} 
\begin{rm}
({\em A distinct perspective) }
The connection between the MGT equation with wave equations with memory has been pointed out in 
the recent independent works \cite{pata1} and \cite{bucci-pan_arxiv2017}.
The critical role of $\gamma$ as a threshold for uniform stability is revisited and recovered in 
\cite{pata1} via the analysis of a corresponding viscoelastic equation. 
It is apparent that appropriate compatibility conditions on initial data must be assumed,
in order to study the third order (in time) equation by using theories pertaining to 
wave equations with a non-local term.

And yet, the perspective of equations with memory opens a distinct avenue of investigation
of the (interior and trace) regularity properties of the corresponding solutions, fruitfully explored in \cite{bucci-pan_arxiv2017} -- as well as, possibly, of other control-theoretic properties.
In this connection, we mention the paper \cite{pandolfi-TAC_2018}, which provides an analysis
of the LQ problem and Riccati equations for finite dimensional systems with memory. 
\end{rm} 
\end{remark}



\subsection{Domain of the generator} \label{ss:domain-generator}
We give an explicit description of the natural domain of the the (free) dynamics generator $A$ introduced in
\eqref{e:generator}: given the state space $Y = H^1(\Omega) \times H^1(\Omega) \times L^2(\Omega)$, 
one has
\begin{equation*}
\begin{split}
y\in \cD(A) \Longleftrightarrow \; & y\in \Big\{y =(y_1,y_2,y_3)  \in Y\colon y_3 \in H^1(\Omega)\,,
\\[1mm]
& \myspace 
c^2 y_1 + b y_2 + N_1 (c N_1^* \cA y_2 + \frac{b}{c} N_1^* \cA y_3)\in \cD(\cA)\Big\}\,. 
\end{split}
\end{equation*}
From the PDE vriewpoint the above corresponds to 
\begin{equation*}
\begin{split}
y\in \cD(A) \Longleftrightarrow \; & y\in \Big\{y \in [H^1(\Omega)]^3\colon \Delta (c^2 y_1 + b y_2) \in L^2(\Omega)\,, 
\;\Dn(c^2 y_1 + b y_2 ) =0 \;\textrm{on $\Gamma_0$}\,,
\\[1mm]
& \myspace 
c\Dn (c^2 y_1 + b y_2 )\Big|_{\Gamma_1} = - \Big[c^2y_2+ b y_3\Big]\Big|_{\Gamma_1}
\;\textrm{on $\Gamma_1$}\Big\}\,.
\end{split}
\end{equation*}
Notice that by a standard variational argument the normal derivatives are first well defined on 
$H^{-1/2}(\Gamma)$. 
Then, the $H^{1/2}(\Gamma)$-regularity of $y_i$, $i =1,2,3$, along with elliptic theory gives
\begin{equation}\label{e:domain-pde}
\begin{split}
& \cD(A) = \Big\{y \in [H^1(\Omega)]^3\colon (c^2 y_1 + b y_2) \in H^2(\Omega)\,, 
\;\Dn(c^2 y_1 + b y_2 ) =0 \;\textrm{on $\Gamma_0$}\,,
\\[1mm]
& \myspace 
c\Dn (c^2 y_1 + b y_2 )\Big|_{\Gamma_1} = - \Big[c^2y_2+ b y_3\Big]\Big|_{\Gamma_1}
\;\textrm{on $\Gamma_1$}\Big\}\,.
\end{split}
\end{equation}
We also note that the resolvent of $A$ is not compact, which is important to be pointed out.


\subsection{The SMGT equation subject to smooth controls}
Now let us turn our attention to the controlled (abstract) equation \eqref{e:controlled-eq} corresponding to the BVP \eqref{e:bvp-0} and to its reformulation as the first-order control system \eqref{e:1order-system}.
This system produces readily a solution formula, assuming that $g\in H^1(0,T;L^2(\Gamma_0))$:
the following Proposition provides a rigorous justification.
  
\begin{proposition} \label{p:firstorder-controlled-wellposed}
Assume that $g\in H^1(0,T;L^2(\Gamma_0))$.
The boundary value problem \eqref{e:bvp-0} for the SMGT equation can be recast as the (third order in time) abstract equation \eqref{e:controlled-eq}; equivalently, it is rewritten as a first order abstract system \eqref{e:1order-system} that is
\begin{equation} \label{e:state-eq_0}
y'=Ay+B_0g+B_1g_t\,;
\end{equation}
$y$ denotes the state variable $(u,u_t,u_{tt})$ and $g$ is the control variable,
while the linear operators $A$ and $B_i$ satisfy the following analytical properties.
\begin{enumerate}

\item[i)] 
the operator $A$ which describes the free dynamics, detailed in \eqref{e:generator},
is the generator of a $C_0$-semigroup $\{e^{At}\}_{t\ge 0}$ on the function space
$Y=H^1(\Omega)\times H^1(\Omega)\times L^2(\Omega)$, with domain 
$\cD(A)$ given in \eqref{e:domain-pde};

\item[ii)]
the control operators $B_i$, $i=0,1$ defined in \eqref{e:input-operators} satisfy 
$B_i\in \cL(U,[\cD(A^*)]')$. 
\end{enumerate}
Then, the third order equation \eqref{e:controlled-eq} is understood on the 
extrapolation space $[\cD(\cA)]'$.

\end{proposition}

\begin{proof}
Since the Neumann maps $N_i$ defined in \eqref{e:N_i} enjoy the regularity in \eqref{e:neumann-regularity}, 
that is $N_i \in \cL(L^2(\Gamma_i),\cD(\cA^{3/4-\sigma}))$, we accordingly have 
that the distributional range of the control maps $B_i$ is such that 
\begin{equation*}
\cR(B_i) \subset  \{0\}\times \{0\} \times [\cD(\cA^{1/4 + \sigma})]'\,.
\end{equation*}
To see this, just recall the explicit form of the input operators $B_0$ in \eqref{e:input-operators},
which gives  
\begin{equation*}
\begin{split}
|(B_0 g,y)|_Y&=|(c^2 (\cA+I)N_0 g,y)|_Y 
=c^2 |(g,y_3|_{\Gamma_0})_{L^2(\Gamma_0)}
=c^2 \,|y_3|_{H^{1/2+2\sigma}(\Omega)} |g|_{L^2(\Gamma_0)}
\\[1mm]
& =c^2\,|\cA^{1/4+\sigma} y_3 |_{L^2(\Omega)} |g|_{L^2(\Gamma_0)}
\end{split}
\end{equation*}
which proves that there exists a positive constant $C$ such that 
\begin{equation*}
|(B_i g,y)|_Y|\le C\,|\cA^{1/4+\sigma} y_3 |_{L^2(\Omega)} |g|_{L^2(\Gamma_0)}
\le C\,|A y|_Y \,|g|_{L^2(\Gamma_0)}\,, \qquad i=0,1\,,
\end{equation*}
since $B_1=b/{c^2} B_0$.

By using interpolation trace results, a stronger inequality is obtained:
for any $\epsilon > 0 $ one has 
\begin{equation*}
(B_i g, y)_Y\le C \,|Ay|_Y^{1/2}|y|_Y^{1/2} |g|_{L^2(\Gamma_i)}
\le  \big(\epsilon |Ay|_Y + C_\epsilon |y|_Y\big)\,|g|_{L^2(\Gamma_i)}
\end{equation*}
which gives 
\begin{equation*}
|B_i^*y|_{L_2(\Gamma_i)} \le \epsilon |Ay|_Y + C_{\epsilon} |y|_Y \qquad  \forall \epsilon > 0\,.
\end{equation*}
\end{proof}

In view of Proposition~\ref{p:firstorder-controlled-wellposed} (hence, still under the assumption
$g\in H^1(0,T,U)$), semigroup theory yields a first input-to-state formula in the extrapolation space $[\cD(A^*)]'$.

\begin{corollary}\label{c:state-in-extrapolation}
For any initial state $y_0 \in [\cD(A^*)]'$ and any control $g\in H^1(0,T,U)$,
the control system \eqref{e:state-eq_0} has a unique mild solution $y\in C([0,T];[\cD(A^*)]')$
given by 
\begin{equation}\label{e:sln_0}
\begin{split}
y(t) &= e^{At} y(0) + \int_0^t e^{A(t-s)} \big(B_0 g(s)+B_1 g_t(s)\big)\,ds=
\\[1mm]
&= e^{At} y(0) + \int_0^t e^{A(t-s)} B_0 \Big(g(s)  + \frac{b}{c^2}g_t(s)\Big)\,ds\,.
\end{split}
\end{equation}

\end{corollary}


\section{The control problem. Main results} \label{s:auxiliary-and-riccati}
If the cost functional \eqref{e:cost-funct} penalized (quadratically) the time derivative of
the control function, we might choose as space of admissible controls 
$\cU=H^1(0,T;L^2(\Gamma_0))$, and the obtained semigroup solution formula \eqref{e:sln_0}
as the state equation.
Remember however that we seek to minimize the functional \eqref{e:cost-funct} over all controls $g$ which belong to $L^2(0,T;L^2(\Gamma_0))$, where the acoustic pressure $u$ satisfies the IBVP \eqref{e:ibvp-1}.
Hence, in this Section we first derive from \eqref{e:sln_0} a solution formula which requires
controls which just belong to $L^2(0,T;L^2(\Gamma_0))$ and are continuous at time $t=0$;
this is done by an elementary integration (in time) by parts.
Then, following an idea proposed in \cite{LLP} and \cite{LPT}, we introduce an (auxiliary) optimal control
problem associated to an equation depending on a parameter $g_0\in L^2(\Gamma_0)=:U$.
The main result pertaining to the auxiliary problem and the connection with the original one 
are stated collectively in the section. The respective proofs are the subject of the
subsequent two sections.

\subsection{Control problem with the observation} 
Our next step is to provide a representation formula for the solutions to the controlled dynamics 
by assuming that controls belong to $L^2(0,T;U)$. 
This is done, as usual, integrating by parts (in a dual space) and exploiting the structure of the domain
of the generator. 

\begin{lemma}\label{l:rep}
Given an initial state $y_0\in [\cD({A^2}^*)]'$ and any control function $g\in C([0,T;U)$, the solution
to the original control system \eqref{e:1order-system}, represented via the
input-to-state formula \eqref{e:sln_0}, is equivalently given by 
\begin{equation}\label{e:eq-for-U}
y(t) =  e^{At}[y_0 - B_1g(0)] + Lg(t)\,,
\end{equation}
with 
\begin{equation}\label{e:input-to-state-map}
\begin{split}
(Lg)(t) &= B_1 g(t) + (L_0 g)(t)\,,
\\[1mm]
(L_0 g)(t) &= \int_0^t e^{A(t-s)} B_0 g(s) ds  + A\int_0^t e^{A(t-s)} B_1 g(s) ds\,.
\end{split}
\end{equation}
The map $(y_0,g) \rightarrow y(\cdot)$ is bounded from 
$[\cD({A^2}^*)]' \times C([0,T; U) \rightarrow C([0,T;[\cD({A^2}^*)]')$.
\end{lemma}

\begin{proof}
The novel representation formula \eqref{e:eq-for-U} is easily established integrating by parts in 
\eqref{e:sln_0}; what we need to justify rigorously is the claimed regularity. 
We know already that $e^{At}$ generates a $C_0$-semigroup on $[\cD({A^2}^*)]'$, and that 
$A^{-1} B_i \in L(U,Y)$.
Then, it suffices to analyze the regularity of the operator $L_0$ in \eqref{e:input-to-state-map},
which depends on the one of the operator $AB_1$.  
Recalling the definitions of $A$ and $B_1$, it is easily seen that 
\begin{equation*}
AB_1 =b
\begin{pmatrix}
0 
\\[1mm]
(\cA+I)N_0
\\[1mm]
- \alpha (\cA+I)N_0
\end{pmatrix} 
\end{equation*}
where we have used that the distributions on $\Gamma_0$ and $\Gamma_1$ have disjoint support.
This implies that the contribution of the operator $N_1$ in the definition of $A$, when applied to $B_1$,  produces the zero element. 
As a consequence, we obtain that 
\begin{equation*}
\cR(AB_1) \subset \{0\} \times [\cD(\cA^{1/4+\epsilon}]' \times [\cD(\cA^{1/4+\epsilon}]' 
\subset [\cD(A^{2*})]'\,,
\end{equation*}
which gives the desired conclusion.
\end{proof}

Observe that -- just like in the works \cite{LLP} and \cite{LPT} -- the drawback of the chosen approach
is that the space regularity of the state function gets worse.
Moreover, in contrast with the dynamics under investigation therein, whose underlying semigroup is 
analytic, we are dealing with a purely hyperbolic problem.
\\
On the other hand, recall that the goal is to minimize the $L^2(\Omega)$-norm of the acoustic pressure,
described by the state variable $u$, that is the first component of the state variabile $y$.
By setting $u^d=0$ in \eqref{e:cost-funct} just for the sake of simplicity, 
the cost functional is abstractly rewritten as
\begin{equation}\label{J}
J(g) =  \int_0^T \|R y\|_Y^2 \,dt + \int_0^T \|g\|_U^2\,dt\,,
\end{equation}
where $U$ denotes the control space, i.e. $U=L^2(\Gamma_0)$, and the observation operator
$R$ is acting as follows: for any $y = [y_1,y_2,y_3]^T$, it holds 
\begin{equation}\label{e:def-observation-op}
R y = \begin{pmatrix}
\cA^{-1/2}y_1 
\\[1mm]
0
\\[1mm]
0
\end{pmatrix}\,.
\end{equation}
In fact, after identifying $H^1(\Omega)$ with $\cD(\cA^{1/2})$, we see that
\begin{equation*}
\|Ry\|_Y = \|\cA^{1/2} \cA^{-1/2} y_1\|_{L^2(\Omega)} = |y_1|_{L^2(\Omega)}\,.
\end{equation*}
Thus, the simple -- and yet natural -- quadratic functional taken into consideration,
attributes to the observation operator $R$ a very special structure and an {\em intrinsic} strong smoothing effect.
The improved regularity of the observed states enables us to pursue an adaptation of 
the theory developed in \cite[Vol.~II]{redbook} in the study of hyperbolic-like PDE's
with boundary or point control actions and ``smoothing'' observations.   

\medskip

\subsection{Main Results}
In this subsection we shall formulate the main results, while the proofs are relegated to the next section. 
We shall begin with a negative result.

Consider the following minimization problem.

\begin{problem} \label{p:pbm_0}
For any $y_0 \in Y$, minimize the cost functional \eqref{J} over all controls $L^2((0,T)\times \Gamma_0)$, 
where $y(\cdot)$ satisfies the controlled equation \eqref{e:eq-for-U}.
\end{problem}

\begin{theorem} \label{l:neg} 
If the initial state $y_0$ belongs to $\cR(B_1)$, then Problem \ref{p:pbm_0}  does not  have a solution. 
\end{theorem}


Given this negative result, one might wonder what are the additional constraints which render the problem solvable. 
The proof of the negative result (cf.~\cite{LPT}) reveals that the issue is in singularity of control, as
the ``candidate'' to be the optimal control is no longer in the space $L^2(0,T; U)$.
(This depends upon the appearance of a (time-)trace operator -- intrisincally uncloseable -- in the definition of the state.) 

In view of the above, we shall consider an input-to-state formula depending on a given parameter $g_0 \in U$, that is 
\begin{equation}\label{g0}
y_{g_0}(t) =  e^{At}(y_0 - B_1g_0) + Lg(t)\,,
\end{equation}
with $L$ defined in \eqref{e:input-to-state-map}. This idea has been developed in \cite{LLP,LPT}.
When $g(0) =g_0$ the above controlled dynamics coincides with the one given by \eqref{e:eq-for-U}.
With \eqref{g0} we associate the same cost functional \eqref{J}.
A new  (`extended') optimal control problem is formulated as follows.

\begin{problem} \label{p:pbm_1-parameter}
For any $y_0 \in [\cD({A^*}^2)]'$, $g_0\in U$, minimize the cost functional \eqref{J} overall controls $g\in L^2((0,T)\times \Gamma_0)$, with $y$ subject to \eqref{g0}.
\end{problem}
For this problem the following results holds true.

\begin{theorem}\label{l:pos} 
The optimization Problem \ref{p:pbm_1-parameter} has a unique solution 
$\hat{g}_{g_0} \in L^2(0,T;U)$.
Its corresponding optimal trajectory satisfies
\begin{equation}\label{e:memberships} 
\hat{y}_{g_0} \in C([0,T];[\cD({A^*}^2]')\,, 
\quad
R\hat{y}_{g_0} \in C([0,T];Y)\,.
\end{equation} 
\end{theorem}

The first main result of the paper establishes the feedback synthesis of the optimal control referred to in Theorem~\ref{l:pos}. 
For clarity of the exposition, we shall take $u_d =0$. 

\begin{theorem}\label{T0}
With reference to the minimization Problem \ref{p:pbm_1-parameter}, the following 
statements are valid.

\begin{enumerate}

\item[i)] (Partial regularity)
For any $y_0 \in [\cD({A^2}^*)]'$, and any $g_0 \in U$, the unique optimal control 
$\hat{g}_{g_0}$ belongs to $C([0,T];U]$, and produces the output 
$R\hat{y}_{g_0} \in C([0,T];Y)$.
 
\item[ii)] (Riccati Equation)  
For every $t \in [0,T]$, there exists a self-adjoint positive operator $P(t)$ on $L(Y)$, 
whose regularity is as follows,
\begin{equation*}
A^* P(t)  \in \cL(Y)\,, \quad B_1^* A^* P(t) \in \cL(Y,U) \; \textrm{continuously in time,}
\end{equation*}
and which satisfies the following (non-standard) Riccati equation:  
\begin{equation}\label{e:RE-0}
\begin{split}
& \frac{d}{dt}(P(t) y, w)_Y +(Ay, P(t) w)_Y + (P(t) y, Aw)_Y + (Ry, Rw)_Y = 
\\[1mm] 
& \qquad =((B_0^* + B_1^* A^*) P(t)y,([B_0^* + B_1 A^*) P(t) w)_{U} 
\quad \textrm{for all $y, w\in \cD(A)$}
\end{split}
\end{equation}
with terminal condition $P(T) =0$.
The equation \eqref{e:RE-0} actually extends to all $y, w \in Y$.

\item[iii)] (Feedback synthesis) 
The optimal control $\hat{g}_{g_0}(\cdot)$ has the following feedback representation:
\begin{equation*} 
\hat{g}_{g_0}(t) = - \big(I - [B_0^* + B_1^* A^*] P(t) B_1 \big)^{-1} 
[B_0^* + B_1^* A^* ] P(t) \hat{y}_{g_0}(t)\,,
\end{equation*}
where the operator $G(t)=I - [B_0^* + B_1^* A^*] P(t) B_1$ is boundedly invertible on $U$ for each 
$t \in [0,T]$.

\end{enumerate}

\end{theorem}
From the structure of the Riccati equation \eqref{e:RE-0}, along with the space regularity of the operator 
$P(t)$ asserted in Theorem \ref{T0}, some additional regularity of the operator $P(t)$ follows. 
\begin{corollary}
The Riccati operator $P(t)$ is time differentiable from $Y$ into itself. 
More precisely, the operator $ \frac{d}{dt} P(t)\colon  Y \rightarrow  C([0,T];Y)$ is bounded. 
\end{corollary}

\begin{remark}
\begin{rm}
We note that the Riccati equation \eqref{e:RE-0} is termed {\em non-standard} 
(already in \cite{LPT}) because of the special structure of its quadratic term. 
This feature results from the lack of coercivity in the functional cost, a cause for singularity of the minimization problem.
Then, the feedback formula which allows the synthesis of the optimal control of Problem~\ref{p:pbm_1-parameter}
involves the inverse of certain operator defined on the control space $U$.
Invertibility of the said operator is an issue already encountered in \cite{LLP} and \cite{LPT}: however, differently from those studies, in the present case we cannot appeal to the analyticity of the semigroup underlying the controlled dynamics.
\end{rm}
\end{remark}

Theorem \ref{T0} provides the optimal control and the optimal synthesis for the input-state dynamics
\eqref{g0}, given $y_0$ and the parameter $g_0$. 
One aims then at exploring the relation between the parameter $g_0$ with the optimal control $\hat{g}$,
which is known from Theorem~\ref{T0} to be continuous on $[0,T]$. 
Thus, a question of major concern is whether the parameter $g_0\in U$ can be selected 
in order that $\hat{g}(0) = g_0$. 
The validity of this property will prove the equivalence of the state description in \eqref{e:eq-for-U} with the one in \eqref{g0}, thereby ensuring that the latter system corresponds to the original PDE model. 
The answer to this question is positive, as asserted by the Theorem below. 

\begin{theorem}\label{T:1}
The operator $[I + G(0)B_1]$ is bounded invertible on $U$; in particular, $[I + G(0)B_1]^{-1} \in \cL(U)$. 
By choosing $g_0 = [I + G(0) B_1]^{-1} G(0) y_0$, one obtains that
\begin{equation*}
\hat{y}(t) = e^{At }[ y_0 -B_1 \hat{g}(0) ] + (L \hat{g})(t)\,,
\end{equation*} 
so that the original dynamics \eqref{e:eq-for-U} coincides with the one in \eqref{g0}. 
Moreover, the obtained $\hat{g}$ is continuous in time, i.e. $\hat{g}\in C([0,T];U)$. 
\end{theorem}

Forcing the original model with continuity of the control at the origin may compromise  the optimality. 
Instead, the additional `player' $g_0 \in U$ is advantageous from the optimality point of view. 
While we know that in general there is no optimal control in the class of $L^2(0,T;U)$ functions 
(cf.~Theorem~\ref{l:neg}), reformulating the solution formula as in \eqref{g0}, with an additional
parameter, gives additional possibilities for  optimization with respect to the parameter.  

\begin{theorem}\label{T:2}
Let $U_0 \subset U$ be a bounded and weakly closed set in $U$. 
Then, there exists a $g^* \in U_0$ such that the resulting control $\hat{g}_{g^*}$ attains the infimum of the functional $J(g)$ with respect to $g_0 \in U_0$, $g \in L^2(0,T;U)$ and $y$ satisfying (\ref{g0}).
Moreover, the following characterization holds true: either $g^*$ is such that 
$y_0 - B_1 g^* \in \ker(B_1^* P(0))$, or $g^* \in \partial U_0$.
\end{theorem}


\begin{remark}
\begin{rm}
Note that the optimal control of Theorem~\ref{T:2} provides control which is in a larger space than just
$L^2(0,T; U)$. 
This is singular control.
The corresponding state is described by \eqref{g0} and it satisfies $R\hat{y}_{\hat{g}_{g^*}} \in C([0,T]; Y)$.
\end{rm}
\end{remark}

It is important to note that from both the point of view of applications as well as of mathematical developments, it is significant to have two versions of optimal solutions corresponding to two different formulations of the input-state map. 
If one is to develop nonlinear versions of the problem, where regularity of controls and of the states is of paramount importance, the first version in Theorem~\ref{T:1} is the most relevant. 
However, from the point of view of automatic control -- where discontinuous inputs are feasible and lead to `better' optimization solutions --, Theorem~\ref{T:2} becomes more relevant. 
Clearly, further discussion of the topic along with relevant examples is appropriate and desirable. 

\begin{remark}
\begin{rm}
There are several open problems sparked off by the present work. We name but a few.
 
\begin{enumerate}
\item[i)]
Extension of the theory to more general observation operators $R$. 
However, it is clear that $R$ should display some kind of smoothing effect. 
Moreover, the structure of the problem -- namely, an appropriate interplay between control and observation
operators -- will need to be carefully chosen, in order that the optimal ($L^2$) solution does exist. 

\item[ii)]
The infinite horizon LQ problem in both the stable and the critical case. 
It is expected that under suitable geometric conditions imposed on $\Gamma_1$ one could guarantee 
solvability of the optimization problem, along with a feedback synthesis of the optimal control. 

\item[iii)]
Application of the previous result to the feedback control of the nonlinear equation. 
A local theory for small initial data should emerge, while the feedback control should provide a stabilizing effect on the nonlinear dynamics. 

\end{enumerate}
\end{rm}
\end{remark}

The remaining parts of the paper are devoted to proofs of  four Theorems. 

\section{Proofs of Theorems \ref{l:neg}, \ref{l:pos}} \label{s:proofs_1}
We point out at the outset that the main challenge in proving the stated results is to be able to `run' the dynamics on much larger dual spaces, still preserving the invariance 
of the said dynamics. 
The following Proposition singles out some basic regularity and structural properties 
pertaining to the observation operator $R$. 

\begin{proposition}\label{p:R}
The observation $R$ satisfies the following properties. 
\begin{itemize}
\item
$R \colon Y \rightarrow \cD(\cA) \times \{0\} \times \{0\}$ is bounded;
\item
$R \in \cL(Y,\cD(A))$; 
\item
$R = R^* $ on $Y$, hence $R \in \cL([\cD(A^*)]',Y)$.
\end{itemize}

\end{proposition}

\begin{proof}
For the first statement, take $y\in Y$: then $y_1 \in \cD(\cA^{1/2})$, and since
$Ry =(\cA^{-1/2} y_1, 0, 0)^T$ we obtain $\cA^{-1/2} y_1 \in \cD(\cA)$. 

The second  statement follows from the calculation with $y \in Y$
\begin{equation*}
A R y = [0,0,-\tau^{-1} c^2 \cA^{1/2} y_1]^T \in Y 
\end{equation*}
We also note that $A^{-1} \in \cL(Y,[\cD(\cA^{1/2})]^3)$. 
The third statement follows from direct calculations using the inner product in $Y$. 

The fourth statement follows combining the third with the second one. 
\end{proof}

\subsection{Properties of the input-to-output map}
The following Lemma captures  
a set of functional-analytic properties pertaining to appropriate combination of the involved abstract operators -- namely, the dynamics, control and observation operators --, which will play a major role in the proof of well-posedness of generalized differential/integral Riccati equations, eventually leading to solvability of the optimal control problem.

\begin{lemma} \label{l:abstract-basics}
Let $A$, $B_i$ and $R$ the dynamics, control, observation operators defined by
\eqref{e:generator}, \eqref{e:input-operators}, \eqref{e:def-observation-op}, respectively.
Then, 
\begin{enumerate}
\item[i)]
$RA^2$ can be extended to a {\em bounded} operator on the state space $Y$;
\item[ii)]
$RB_1=0$;
\item[iii)]
$(I+A)^{-1}B_i$ are bounded and compact operators from $L^2(\Gamma_i)$ into $Y$, $i=0,1$.
\end{enumerate}
\end{lemma}

\begin{proof}
i) We take an element $y=(y_1,y_2,y_3)$ initially assumed in $\cD(A^2)$, and compute
\begin{equation*}
\begin{split}
& A^2y = A(Ay)= 
\\[1mm]
& \quad= A
\begin{pmatrix}
y_2
\\[1mm]
y_3
\\[1mm]
- c^2\cA y_1- [b \cA+c(\cA+I)N_1N_1^*(\cA+I)]y_2
-\big[\alpha I +\frac{b}{c}(\cA+I)N_1 N_1^*(\cA+I)\big]y_3
\end{pmatrix}
\\[1mm]
& \quad = \begin{pmatrix}
y_3
\\[1mm]
\ldots\ldots\ldots
\\[1mm]
\ldots\ldots\ldots\ldots\ldots\ldots
\end{pmatrix}
\end{split}
\end{equation*}
where the second and third component of $A^2y$ are unspecified, owing to the 
structure of the observation operator $R$ to be applied.
Consequently, 
\begin{equation*}
R\, A^2y = \begin{pmatrix}
(I+\cA)^{-1/2} y_3 \\ 0 \\ 0
\end{pmatrix}
\end{equation*}
and 
\begin{equation*}
\|R\, A^2y\|_Y = \left\| \begin{pmatrix} (I+\cA)^{-1/2} y_3 \\ 0 \\ 0 \end{pmatrix}\right\|_Y
= \big\|(I+\cA)^{1/2}(I+\cA)^{-1/2}y_3\big\|=\|y_3\|_{L^2(\Omega)}
\end{equation*}

\smallskip
\noindent
ii) It is immediately verified that for any $h\in L^2(0,T;L^2(\Gamma_1))$ 
\begin{equation*}
R\,B_1 h=R \begin{pmatrix}
0 
\\[1mm]
0 
\\[1mm]
b (\cA+I)N_1\,h
\end{pmatrix}
=(I+\cA)^{-1/2}\,0=0\,.
\end{equation*}

\smallskip
\noindent
iii) It is clear that the resolvent $(I+A)^{-1}$ is not compact.
However, we have
\begin{equation}
(I+A)^{-1} B_0  = c^2 \begin{pmatrix} N_0 \\ 0 \\ 0 \end{pmatrix}\,, 
\qquad
(I+A)^{-1} B_1 = \frac{b}{c^2} (I+A)^{-1} B_0=
bc^{-2}\begin{pmatrix} N_0 \\ 0 \\ 0 \end{pmatrix}
\end{equation}
and because $\cR(N_0) \subset H^{3/2}(\Omega)$, the operators $A^{-1} B_i$
are not only bounded from $L_2(\Gamma_0) \rightarrow  Y$, but also compact. 

\end{proof}

The following Lemma pertains to the regularity of the map $RL_0$.

\begin{lemma}\label{l:L}
Let $L_0$ be the operator defined by \eqref{e:input-to-state-map}.
Then 
\begin{itemize}
\item
$RL_0$ is a compact operator from $L^2(0,T;L^2(\Gamma_0))$ into $C([0,T];Y)$.
\item
$R e^{A\cdot} B_i \colon L^2(0,T;L^2(\Gamma_0)) \rightarrow  C([0,T];Y)$, $i=0,1$,
are compact. 
\end{itemize}

\end{lemma}

\begin{proof}
The first statement follows computing 
\begin{equation}
\begin{split}
R L_0\colon g\longmapsto R L_0 g 
& = R \int_0^t e^{A(t-s)} B_0 g(s) ds  - R A\int_0^t e^{A(t-s)} B_1 g(s) ds
\\[1mm]
& = R A^2 \int_0^t e^{A(t-s)} A^{-2} B_0 g(s) - R A^2 \int_0^t e^{A(t-s)} A^{-1} B_1 g(s)ds
\end{split}
\end{equation}
in view of Lemma~\ref{l:abstract-basics}, combined with Aubin-Simon compactness criterion. 

The second statement follows rewriting $R e^{At } B_0$ as follows,
\begin{equation*}
R e^{At } B_0 = R A e^{At} A^{-1} B_0\,,
\end{equation*}
where $RA\in \cL(Y)$ and $A^{-1}B_0 \colon U \rightarrow Y$ compactly. 
The strong additional regularity $R A^2 \in \cL(Y) $ allows to handle 
the time derivative 
\begin{equation*}
\frac{d}{dt} R Ae^{At } A^{-1} B_0 = RA^2 e^{At} A^{-1} B_0 \in \cL(U,Y)\,,
\end{equation*}
as needed for the applicability of the Aubin-Simon compactness criterion. 
\end{proof}

\subsection{Proof of Theorem~\ref{l:neg}} 
We will denote by $J(g)$ the cost functional $J(g,y)$, where $y(\cdot)=y(\cdot;g)$ corresponds to the state variable given by \eqref{e:eq-for-U}. 
Take $y_0 \in \cR(B_1)\in [\cD(A^*)]'$ and select a sequence of controls $g_n \in H^1(0,T; U )$
such that  
\begin{itemize}
\item[i)]
$B_1 g_n(0) = y_0$, 
\item[ii)]
$g_n \rightarrow 0$ in $L^2(0,Y;U)$. 
\end{itemize}
Then, with $y_n(t) =y_n(t,g_n) = e^{At} \big(y_0 - B_1 g_n(0)\big) + (L_0 g_n )(t) 
+ B_1 g_n (t)$  
we have 
\begin{equation*}
Ry_n = R L_0 g_n \longrightarrow 0 \quad \textrm{in $L^2(0,T;Y)$,} 
\end{equation*}
on the strength of Lemma~\ref{l:L}. 
Consequently, $J(g_n) \rightarrow 0$.

Since $g_n \rightarrow 0$ in $L^2(0,T;U )$, we turn to 
$J(0) = \int_0^T |R e^{At} y_0 |_Y^2 dt > 0$, 
which combined with $g_n \rightarrow 0$ contradicts the existence of a minimizer. 

\subsection{Proof of Theorem \ref{l:pos}} 
The argument is in principle standard, as it is based on proving  weak lower semicontinuity
of the cost functional. 
Thus, the challenge is to establish appropriate regularity of the input-to-state map,
which is not obvious in view of the high unboundedness of the control input operators. 
However, this is possible exploiting the smoothing effect of the observation operator
as well as the properties specifically established for the input-to-output map (cf.~Lemma~\ref{l:L}). 
To wit: for a given $g_0 \in U$ consider a minimizing sequence $g_n \in L^2(0,T;U)$,
so that $J(g_n) \rightarrow d = \inf_{g_n \in  L^2(0,T;U)} J(g_n)$.
Then, coercivity of the cost in $L^2(0,T;U)$ gives the bound $\|g_n\|_{L_2(0,T;U)} \le M$
which implies that
\begin{equation} 
g_n \rightarrow g \quad \text{weakly in $L^2(0,T;U)$.} 
\end{equation}
We also have 
\begin{equation*}
R y_n(t) = Re^{At} (y_0 - B_1 g_0) + (RL_0 g_n)(t)\,.
\end{equation*}
On the strength of Lemma~\ref{l:abstract-basics} and Lemma \ref{l:L}, for a subsequence
-- denoted by the same symbol -- it follows $R L_0 g_n \rightarrow R L_0 g$ in $L^2(0,T;Y)$.
In addition, $R e^{At} B_1 =  R A e^{At}A^{-1} B_1$ is bounded from $L^2(0,T;U)$ into
$L^2(0,T;Y)$. 
This implies the weak lower semicontinuity of $J(g)$, along with $J(g) \le d $, which proves optimality.
The regularity in \eqref{e:memberships} pertaining to the observed optimal state,
follows in view of the obtained regularity of the three summands in 
\begin{equation*}
R y(t) = Re^{At} (y_0 - B_1 g_0) + (RL_0 g)(t)\,,
\end{equation*}
where in particular $Re^{At} y_0=RA^2 e^{At} A^{-2}y_0\in C([0,T];Y)$ for any 
$y_0\in [\cD((A^2)^*)]'$, 
thanks to the property i) of Lemma~\ref{l:abstract-basics}. 

\section{Proof of Theorem \ref{T0}} \label{s:proofs_2}
Given the the solution formula \eqref{g0}, with the input-to-state map $L$ defined
in \eqref{e:input-to-state-map}, let us consider the dynamics 
\begin{equation}\label{alpha}
y_{\alpha}(t) = e^{At} \alpha + (Lg)(t) 
\end{equation}
depending on the parameter $\alpha \in [\cD(A^*)]'$.
This choice is justified by $B_1 g_0 \in [\cD(A^*)]'$ for $g_0\in U$.
(We note that $y_{g_0}(\cdot)$ has been used to denote the function in \eqref{g0}, with emphasis on the dependence of $y$ on $g_0\in U$, beside to $y_0$. 
In the present section, although with a certain abuse of notation, with $y_{\alpha}(\cdot)$
we shall be always referring to the ``full'' parameter $\alpha$, rather than to its component $g_0$.) 
Recall that  
\begin{equation}\label{A-1B}
A^{-1} B_0 g= [\cA^{-1} (\cA + I ) N_0g, 0, 0 ]^T \in 
H^{3/2}(\Omega) \times \{0\} \times \{0\} \subset Y\,. 
\end{equation}
We add that on the strength of \eqref{A-1B} and $A^{-2} AB_1 = A^{-1} B_1$,
one gets  
\begin{equation}\label{L}
L\in \cL(L^2(0,T;U),C([0,T];[\cD(A^*)^2]'))\,.
\end{equation}

The following auxiliary control problem is naturally associated to \eqref{alpha}.

\begin{problem}[\bf Problem $\cP_\alpha$] \label{p:alfa}
For any $\alpha \in [\cD(A^{2*})]'$, minimize the functional 
\begin{equation}\label{Jx}
J(g,y_{\alpha}) =  \int_0^T \|R y_{\alpha}\|_Y^2 \,dt + \int_0^T \|g\|^2_{U}\,dt\,,
\end{equation}
overall controls $g\in L^2(0,T;U)$, with $y_{\alpha}(\cdot)$ solution to \eqref{alpha}.
\end{problem}

Of course, our goal is to obtain the results in the topology of the original spaces $Y$ and $U$.  While this is not possible for the entire control system, 
it turns out that the optimal solution displays an additional regularity that will make it
possible the return to the original state space. 
The corresponding result is formulated below. 
For simplicity of notation we shall set $C(Y) = C([0,T];Y)$ and $L^2(Y) = L_2(0,T;Y)$;
a similar notation will be adopted with $Y$ replaced by $U$. 

\begin{proposition}\label{T}
With reference to the parametrized control Problem~\ref{p:alfa}, the following statements
are valid. 
\begin{enumerate}

\item[i)]
For any $\alpha \in [\cD({A^*}^2)]'$, there exists a unique optimal control
$g^0 (\cdot)\in L^2(0,T;U)$, which additionally satisfies $g^0 \in C([0,T];U)$. 
Moreover, $R y_{\alpha}^0 \in  C[[0,T];Y)$. 

\item[ii)]
There exists a selfadjoint, positive operator $P(t)$ on $\cL(Y)$ with the
following regularity,
\begin{equation*}
A^* P(t)A \in \cL(Y,C(Y))\,, \quad B_1^* A^* P(t) \in \cL(Y,C(U))\,,
\quad P_t \in \cL(Y,C(Y))\,;
\end{equation*}
$P(t)$ satisfies the following (non-standard) Riccati equation, valid for any $y, w \in \cD(A)$:
\begin{equation}\label{Ric}
\begin{split}
& \frac{d}{dt}(P(t)y,w)_Y +(Ay, P(t)w)_Y + (P(t) y, Aw)_Y + (Ry, R\hat{y} )_Y = 
\\[1mm] 
& \; \big( (B_0^* + B_1^* A^*) P y, [I + B_1^* R^* R B_1]^{-1} 
[ (B_0^* + B_1 A^*) P(t)w] \big)_U\,,
\end{split}
\end{equation}
with terminal condition  $P(T) =0$.

\item[iii)]
For every $\alpha  \in \cD({A^2}^*)]'$, the optimal cost  
$J(g^0) =  \min_{g \in L_2(0,T; U )} J(g,y_{\alpha})$ is given by 
$J(g^0)  = (P(0)\alpha, \alpha)_Y$. 

\item[iv)]
The optimal control has the following feedback representation:
\begin{equation*} 
g^0(t) = - \big[I - (B_0^* + B_1^* A^*) P(t) B_1\big]^{-1} 
\big[ (B_0^* + B_1^* A^*)P(t)\big] y_{\alpha}^0(t)\,,
\end{equation*}
where the operator $I - (B_0^* + B_1^* A^*)P(t)B_1$ is boundedly invertible on $U$ 
for each $t \ge 0$.  

\end{enumerate}
\end{proposition}

\ifdefined\xxxxxx
\subsection{Matching the initial condition}. 
We note that for any $x= y_0 - B_1g_0 $ with $y_0 \in [\cD(A^*)]'$ and $g_0 \in U$, 
the corresponding optimal control $g^0 \in C([0,T];U)$. 
The latter follows from Part 1 of Theorem~\ref{T}. 
Therefore, in order to comply with the original model one is asking for the following selection of $g_0$ $g_0 = g^0(0) $. This amounts to 
$$ g^0_{x}(t=0)=g_0, ~x = y_0 - B_1g_0 $$
The above implicit relation is always uniquely solvable for $g_0 \in U $  as shown in \cite{T}. 
In fact, the matching condition amounts to solving 
$$ g_0 = G x = G (y_0 - B_1 g_0 ) $$
$$[I + G B_1] g_0 = Gy_0 $$
where $G\equiv - [ I + B_1^* R^* R B_1]^{-1} [ B_1^* R^* R + ( B_0^* + B_1^* A^* ) P(0) ]$

With the key properties $G \in L([D(A^*)]') \rightarrow U )$ 
and $ [ I + GB_1 ]^{-1}\in L(U ) $ to be shown later. 
Thus we  obtain 
\begin{corollary}
Let $y_0 \in D(A^*)]' $ be given. Consider Problem $\mathcal{P}_x $ with $x = y_0 -B_1g_0$ and $g_0 \in U $ is given by
\begin{equation}\label{g00}
 g_0 =  [ I + GB_1 ]^{-1} G y_0 
 \end{equation}
 Then there exists unique optimal control $g^0\in L_2(0,T; U) $ with the 
 corresponding trajectory  (\ref{e:eq-for-U}) and initial condition  $y^0(0) = y_0 $, such that 
 the results of Theorem \ref{T} holds with $x = y_0 - B_1 g_0 $ .
 \end{corollary}
 
In  other words, by solving parametrized  optimal control with  a given $x = y_0 - B_1 g_0 $ and a parameter  p $g_0 \in U $ we solve a family of  parametrized optimal control problems, which always has a unique solution. 
 The original  dynamics is included in tis family. 
 By selecting $g_0$ according to the matching condition, we make a selection of a problem whose dynamics coincides with the original one.  However, the above does not imply that the constructed optimal control for parametrized  control problem is also optimal for the original problem -when considered within $L_2(U)$ framework for optimal controls.  In fact, the latter may not have optimal solution at all when 
$y_0 \in \cR(B_1) $ -as shown in \cite{LPT}. Thus, the constructed control is suboptimal -but it corresponds 
 the original dynamics. 
 However, if the original problem does have $L_2(U) $ optimal control, then such control coincides with 
 a parametrized control where $g_0$ is selected according to  the matching condition. 
\fi

\subsection{Proof of Proposition \ref{T} }
\ifdefined\xxx

\subsection{Properties of the input-output map}
The following Lemma captures 
a set of functional-analytic properties pertaining to appropriate combination of the involved abstract operators -- namely, the dynamics, control and observation operators --, which will play a major role in the proof of well-posedness of generalized differential/integral Riccati equations, eventually leading to solvability of the optimal control problem.

\begin{lemma} \label{l:abstract-basics}
Let $A$, $B_i$ and $R$ the dynamics, control, observation operators defined by
\eqref{e:generator}, \eqref{e:input-operators}, \eqref{e:def-observation-op} respectively.
Then, we have 
\begin{enumerate}
\item[i)]
$RA^2$ can be extended to a {\em bounded} operator on the state space $Y$;
\item[ii)]
$RB_1=0$;
\item[iii)]
$(I+A)^{-1}B_i$ are bounded and compact operators from $L^2(\Gamma_i)$ into $Y$, $i=0,1$.
\end{enumerate}
\end{lemma}

\begin{proof}
i) We take an element $y=(y_1,y_2,y_3)$ initially assumed in $\cD(A^2)$, and compute
\begin{equation*}
\begin{split}
& A^2y = A(Ay)= 
\\[1mm]
& \quad= A
\begin{pmatrix}
y_2
\\[1mm]
y_3
\\[1mm]
- c^2\cA y_1- [b \cA+c(\cA+I)N_1N_1^*(\cA+I)]y_2
-\big[\alpha I +\frac{b}{c}(\cA+I)N_1 N_1^*(\cA+I)\big]y_3
\end{pmatrix}
\\[1mm]
& \quad = \begin{pmatrix}
y_3
\\[1mm]
\ldots\ldots\ldots
\\[1mm]
\ldots\ldots\ldots\ldots\ldots\ldots
\end{pmatrix}
\end{split}
\end{equation*}
where the second and third component of $A^2y$ are unspecified, owing to the 
structure of the observation operator $R$ to be applied.
Consequently, 
\begin{equation*}
R\, A^2y = \begin{pmatrix}
(I+\cA)^{-1/2} y_3 \\ 0 \\ 0
\end{pmatrix}
\end{equation*}
and 
\begin{equation*}
\|R\, A^2y\|_Y = \left\| \begin{pmatrix} (I+\cA)^{-1/2} y_3 \\ 0 \\ 0 \end{pmatrix}\right\|_Y
= \big\|(I+\cA)^{1/2}(I+\cA)^{-1/2}y_3\big\|=\|y_3\|_{L^2(\Omega)}
\end{equation*}

\smallskip
\noindent
ii) It is immediatly verified that for any $h\in L^2(0,T;L^2(\Gamma_1))$ 
\begin{equation*}
R\,B_1 h=R \begin{pmatrix}
0 
\\[1mm]
0 
\\[1mm]
b (\cA+I)N_1\,h
\end{pmatrix}
=(I+\cA)^{-1/2}\,0=0\,.
\end{equation*}

\smallskip
\noindent
iii) It is clear that the resolvent $(I+A)^{-1}$ is not compact.
However, we have
\begin{equation}
(I+A)^{-1} B_0  = c^2 \begin{pmatrix} N_0 \\ 0 \\ 0 \end{pmatrix}\,, 
\qquad
(I+A)^{-1} B_1 = \frac{b}{c^2} (I+A)^{-1} B_0=
bc^{-2}\begin{pmatrix} N_0 \\ 0 \\ 0 \end{pmatrix}
\end{equation}
and because $\cR(N_0) \subset H^{3/2}(\Omega)$, the operators $A^{-1} B_i$
are not only bounded from $L_2(\Gamma_0) \rightarrow  Y$, but also compact. 

\end{proof}

The following Lemma pertains to regularity of the map $RL_0$, that is

\begin{lemma}
Let $L_0$ be the operator defined by \eqref{e:input-to-state-map}.
Then $RL_0$ is a compact operator from $L^2(0,T;L^2(\Gamma_0))$ into $C([0,T];Y)$.
\end{lemma}

\begin{proof}
Follows from Lemma~\ref{l:abstract-basics} 
followed by
\begin{equation}
\begin{split}
R L_0\colon g\longmapsto R L_0 g 
& = R \int_0^t e^{A(t-s)} B_0 g(s) ds  - R A\int_0^t e^{A(t-s)} B_1 g(s) ds
\\[1mm]
& = R A^2 \int_0^t e^{A(t-s)} A^{-2} B_0 g(s) - R A^2 \int_0^t e^{A(t-s)} A^{-1} B_1 g(s)ds
\end{split}
\end{equation}
and combined with Aubin-Simon compactness criterion. 
\end{proof}
\fi

\subsubsection{The parametrized LQ-problem} 
The starting point is the semigroup solution $y(t)= e^{At} \alpha  + Lg(t)$.
In order to derive the synthesis for the ``enlarged'' problem by introducing a parameter 
$\alpha \in Y$ and later considering the family of control problems depending on a parameter 
$\alpha \in Y \oplus \cR(B_1)$, one needs to develop a dynamics that is invariant on the space compatible with initial data. 

In view of the above, it is essential to extend the action of the semigroup $e^{At}$, originally defined on $Y$, to a larger space which contains $Y \oplus \cR (B_1)$. 
This can be done on the strength of the extended regularity of the operator $B_i$ as acting 
into dual spaces of $\cD(A^*)$. This will be seen below. 
The low regularity of the input-to-state mapping $L$  will force us to  run  the dynamics written below on an even larger space which is $[\cD({A^2}^*)]'$. 
\begin{equation}\label{e:eq-in-U-with-alpha}
y(t) = e^{At} \alpha + Lg(t) = e^{At} \alpha + B_1 g(t) +[ L_0 g](t)  \,,
\end{equation}
It is important to emphasize that $y(0)\ne \alpha$, whereas $y(0)= \alpha+B_1g(0)$. 
Since
\begin{equation*}
A^{-1} B_1g  =  
\begin{pmatrix}b c^{-2} \cA^{-1} (\cA + I ) N_0g
\\
0
\\
0
\end{pmatrix}\,,
\end{equation*}
we have $A^{-1} B_1 \in \cL(U,Y)$ (in fact, compactly). 
This follows from the regularity of the Neumann map where 
$N_0 \in \cL(L^2(\Gamma_0),H^{3/2}(\Omega))$, where $H^{3/2}(\Omega) \subset \cD(\cA^{1/2})$ 
(the latter being a compact embedding). 
We can thus take $\alpha$ in $[\cD({A^2}^*)]'$.
So the dynamics operator with $g \in C([0,T];U)$ will have values in the dual space 
$[\cD({A^2}^*)]'$.

By the same arguments as these used for the proof of Theorem \ref{l:pos} we obtain

\begin{lemma}[Auxiliary optimal control problem] \label{p:auxiliary}
Given $\alpha \in [\cD({A^*}^2]'$,  there exists a control function 
$g^0\in L^2(0,T;U)$ which minimizes the cost functional \eqref{Jx}, where 
$y(\cdot)$ is the solution to \eqref{e:eq-in-U-with-alpha} corresponding to
the control $g(\cdot)$.
\end{lemma}
Our  main goal is to provide a feedback synthesis of the optimal control $g^0$. 

While the  existence of optimal solution for the parametrized problem follows from
Lemma~\ref{p:auxiliary}, in order to provide a (pointwise in time) feedback representation
of the optimal control -- via the optimal cost operator $P(t)$ -- one needs to introduce, for any $s\in [0,T)$, the dynamics described by the equation 
\begin{equation}\label{e:s-eq-in-U-with-alpha}
y(t,s;\alpha) = e^{A(t-s)} \alpha + L_sg(t)\,, \qquad s\le t\le T\,,
\end{equation}
as well as the cost functional
\begin{equation}\label{e:s-cost}
J_{s,T}(g) \equiv \int_s^T \big(\|Ry(t)\|^2_{Y}  + \|g(t)\|^2_U \big)\,dt\,, 
\end{equation}
where as before $y=(u,u_t,u_t)$ and $L_{s,T}$ -- $L_s$, in short -- is the operator defined by 
\begin{equation}\label{e:s-input-to-state-operator}
\{L_sg\}(t)= \int_s^t e^{A(t-\tau)} B_0 g(\tau)\,d\tau 
+ A\int_s^t e^{A(t-\tau)} B_1 g(\tau)\,d\tau +B_1 g(t) \qquad \forall t\in [s,T]\,. 
\end{equation}
(Note that the subscript {\em s} refers to initial time: in order to avoid confusion, the former operator $L_0=L-B_1$ is written $L^0$.)

\begin{lemma}\label{l:regularity-input-to-state}
One has the following basic regularity of the input-to-state map:
\begin{equation*}
L^0_s \; \text{is continuous}\colon L^2(s,T;U) \longrightarrow C([s,T];[\cD({A^*}^2)]')\,,
\end{equation*}
\begin{equation*}
L_s \; \text{is continuous}\colon L^2(s,T;U) \longrightarrow 
L^2(s,T;[\cD(A^*)]')\oplus C([s,T];[\cD({A^*}^2)]')\,,
\end{equation*}
The above regularity improves when  input-to-state map is combined with the observation operator
$R$; indeed, for the operator $RL$ and its adjoint it holds
\begin{equation*}
\begin{split}
& RL_s \; \text{continuous}\colon L^1(s,T;U) \longrightarrow C([s,T];Y)\,;
\\[1mm]
& L_s^*R^* \; \text{continuous}\colon L^1(s,T;Y) \longrightarrow C([s,T];U)\,.
\end{split}
\end{equation*}
In addition, the operator $L_s^*R^*$ satisfies 
\begin{equation*}
L_s^*R^* \; \text{continuous}\colon L^2(s,T;Y) \longrightarrow C([s,T];U)
\end{equation*}
uniformly with respect to $s\in [0,T)$. 
\end{lemma}

\begin{proof}
The regularity of the control-state map follows from the quantified regularity of $B_1$ map which takes boundedly $U$ into $[\cD((A^*)]'$. 
Then the first statement in the Lemma follows from the structure of $L$ operator.
The key in the regularity  control $\rightarrow$ observation operator is the combination of the three properties $A^{-2}B_i\in\cL(U,Y)$, $i=1,2$, $RA^2\in \cL(Y)$, $RB_1=0$.
\end{proof}

\begin{lemma}
With reference to the optimal control problem \eqref{e:s-eq-in-U-with-alpha}-\eqref{e:s-cost},
the following statements are valid:
\begin{enumerate}
\item[i)]
{\bf (Optimal pair). }
Given $\alpha\in [\cD({A^*}^2)]'$, there exists a unique optimal pair
\begin{equation*}
(\hat{y}(t,s;\alpha),\hat{g}(t,s;\alpha)) 
\end{equation*}
for Problem~\ref{p:auxiliary}, with 
\begin{subequations}
\begin{align}
& \hat{g}(t,s;\alpha)=[I+L_s^*R^*RL_s]^{-1}L_s^*R^*Re^{A(\cdot-s)}\alpha\in C([s,T];U)\,,
\\[2mm]
& \hat{y}(t,s;\alpha)= 
e^{A(t-s)}\alpha + \{L_s\hat{g}(\cdot,s;\alpha)\}(t)\in C([s,T];[\cD({A^*}^2)]')\,,
\label{e:basicregularity}\\[2mm]
& R\hat{y}(t,s;\alpha)= [I+RL_sL_s^*R^*]^{-1}Re^{A(\cdot-s)}\alpha \in C([s,T];Y)\,.
\end{align}
\end{subequations}
%
%
\item[ii)]
{\bf (Riccati operator). }
The operator $P(t)\in \cL(Y)$, $t\in [s,T]$, is 
given by
\begin{equation}\label{e:riccati-operator-2}
P(t) \alpha = \int_t^T e^{A^*(\tau-t)}R^*R \hat{y}(\tau,t;\alpha)\,d\tau\,, 
\end{equation} 
The operator $P(t)$ is positive selfadjoint on $Y$, and represents the optimal cost (or Riccati) operator; its regularity properties are detailed separately (cf.~Proposition~\ref{p:Riccati-operator} below).
\item[iii)]
{\bf (Implicit feedback formula). }
The optimal control satisfies
\begin{equation*}
\hat{g}(t,s;\alpha)= -[B_0^*P(t) +B_1^*A^*P(t)]\Phi(t,s)\alpha\,,
\end{equation*}
that is the following implicit equation
\begin{equation*}
\hat{g}(t,s;\alpha)= -[B_0^*P(t) +B_1^*A^*P(t)]\hat{y}(t,s;\alpha)
+[B_0^*P(t) +B_1^*A^*P(t)]B_1\hat{g}(t,s;\alpha)\,,
\end{equation*}
where the operator $\Phi(t,s)$ is defined in \eqref{e:transition}.
\item[iv)]
{\bf (Optimal cost). }
The optimal cost for Problem~\ref{p:auxiliary} is given by
\begin{equation*}
J_s(\hat{g}) =\int_s^T \big(\|R\hat{y}\|^2_Y + |\hat{g}(t)|^2_U \big)\,dt
= \|[I+RL_sL_s^*R^*]^{-1/2}\, Re^{A(\cdot-s)}\alpha\|_{L^2(s,T;Y)}^2  
\end{equation*}
which is rewritten in terms of the optimal cost (or Riccati) operator as follows
\begin{equation}
\begin{split}
J_s(\hat{g}) &=(P(s)\alpha,\alpha)=
\\
&= \big([I+RL_sL_s^*R^*]^{-1}\, Re^{A(\cdot-s)}\alpha,Re^{A(\cdot-s)}\alpha\big)_{L^2(s,T;Y)}\,,
\end{split}
\end{equation}
thereby providing 
\begin{equation}
P(s)\alpha=e^{A^*(\cdot-s)}R^*\,[I+RL_sL^*_sR^*]^{-1}\,Re^{A(\cdot-s)}\alpha 
\quad \forall \alpha \in [\cD({A^*}^{2}]'\,.
\end{equation}

\end{enumerate}

\end{lemma}

\begin{proof}
1. The first statement follows by standard variational arguments applied to the 
LQ-problem (cf.~\cite{redbook}), 
after taking into consideration the regularity of input-output map stated
in the preceding Lemma. 
The formulas for the optimal control, optimal state, observed state are derived
as usual from the optimality conditions. 
The regularity of the optimal quantities follows from the regularity of the map $L$. 
In fact $A^{-2}\alpha \in Y$ gives $R e^{At} \alpha = RA^2 e^{At} A^{-2}\alpha \in C([0,T];Y)$ and by Lemma~\ref{l:regularity-input-to-state}
\begin{equation*}
L_s^* R^* R e^{A\cdot} \alpha\in C([0,T];U)\,.
\end{equation*}

We note that the invertibility of the operator $I + L_s^* R^* R L_s$ on $C([s,T];U)$ follows
combining the self-adjointness and positivity of $L_s^* R^* R L_s$ -- which guarantees the invertibility on $L_2(U)$ -- with boundedness of the latter operator on $C([s,T];U)$. 
A classical bootstrap argument yields the claimed regularity: 
one starts from 
\begin{equation*}
v=[I + L_s^* R^* R L_s ]^{-1} g\,,
\end{equation*}
 with $g \in C(U)$, obtaining first 
$v \in L^2(U)$; then, since $v = - L_s^* R^* R L_s v +g$, the regularity improves to 
$v \in C(U)$. 

A word of caution: while $RL_0 $ is compact on $L^2(U)$, this is no longer the case for $RL$,
owing to the presence of the summand $R B_1$, which is not time compact. 

The regularity of $R\hat{y}(t,s;\alpha)$ is a consequence of the regularity of the operator 
$RL$ in Lemma~\ref{l:regularity-input-to-state}. 
Then, by the optimality condition 
\begin{equation}\label{e:optimal-control-from-optimality}
\hat{g}(t,s;\alpha)=-\{L_s^*R^*[R\hat{y}(\cdot,s;\alpha)]\}(t)\,,
\end{equation}
which combined with the regularity of the operator $L_s^*R^*$ yields continuity (in time)
of the optimal control. 

\noindent
\smallskip
2. All the statements ii)-iv) follow by variational arguments, 
by using the structure of the optimal quantities, once several properties 
that specifically pertain the operators $\Phi(\cdot,\cdot)$ and $P(\cdot)$ are proved. 
These technical results are given in the Propositions which follow next.  
\end{proof}

\begin{remark}
\begin{rm}
A peculiarity of the parametrized minimization problem is that the optimal trajectory
does not satisfy the evolution property. 
(For this reason the Riccati operator and the resulting synthesis cannnot be standard, as certain cancellations do not occur.)
In the next section we study the evolution operator, defined only on a dual (extrapolation) space. 
This is a consequence of the low regularity of the control-to-state map. 
\end{rm}
\end{remark}

\subsubsection{The operator $\Phi(t,s)$}
One of the most critical ingredients of Riccati theory is the evolution operator which
describes controlled dynamics. 
While in the standard theory the evolution operator is constructed directly from the optimal trajectory, this is not the case in singular theory. 
The reason is that such operator will not display the evolution property -- the most fundamental feature. 
For this reason we define evolution differently, as in the formula below. 
  
For any couple $(t,s)$ such that $0\le s\le t\le T$, let 
$\Phi(t,s)\colon [\cD({A^*}^2)]'\rightarrow [\cD({A^*}^2)]'$ defined by
\begin{equation}\label{e:transition}
\Phi(t,s)\alpha := \hat{y}(t,s;\alpha)-B_1\hat{g}(t,s;\alpha)=
e^{A(t-s)} \alpha + \{L^0_s \hat{g}(\cdot,s;\alpha)\}(t)\,. 
\end{equation}
The regularity properties of the operator $\Phi(\cdot,\cdot)$, which a priori
belongs to $\cL([\cD({A^*}^2)]')$ (for $(t,s)$ given), are collected in the following Proposition.

\begin{proposition}\label{p:Phi}
For the operator $\Phi(\cdot,\cdot)$ defined in \eqref{e:transition} 
the following properties are valid:
\begin{enumerate}
\item[i)]
$\Phi(t,t)\alpha =\alpha$ for all $\alpha\in [\cD({A^*}^2)]'$;
\item[ii)]
for any $s, \tau$ with $0\le s\le \tau\le T$, it holds
\begin{equation}\label{e:transition-item2}
Re^{A(\cdot-\tau)}\Phi(\tau,s)\alpha\in C([\tau,T];Y) \qquad \forall \alpha\in [\cD({A^*}^2)]'
\end{equation}
continuously with respect to $\alpha$ and uniformly in $s$ and $\tau$;
\item[iii)]
for any $s, \tau$ with $0\le s\le \tau\le T$, it holds
\begin{equation*}
R\Phi(\cdot,\tau)\,\Phi(\tau,s)\alpha\in C([\tau,T];Y) \qquad \forall \alpha\in [\cD({A^*}^2)]'
\end{equation*} 
continuously with respect to $\alpha$ and uniformly in $s$ and $\tau$;
\item[iv)]
for any $s,\tau, t$ with $0\le s\le \tau\le t\le T$, it holds in $Y$ 
\begin{equation*}
R\Phi(t,\tau)\,\Phi(\tau,s)\alpha= R \Phi(t,s)\alpha \qquad \forall \alpha\in [\cD({A^*}^2)]'
\end{equation*} 
\end{enumerate}

\end{proposition}

\begin{proof}
Since the operator $\Phi(t,s)$ -- as defined above -- has the same algebraic structure as in
the classical LQ-theory, we can treat this operator as an evolution on the dual space to 
$\cD({A^*}^2)$. 
The needed regularity is established by referring to preceding Lemmas: in particular,
to Lemma~\ref{l:regularity-input-to-state}.
The proof of the above properties can be produced along the lines of Lemma~8.3.2.3
and Lemma~8.3.2.4 in \cite{redbook}, on the basis of the powerful facts
$RB_1=0$, $R A^2\in \cL(Y)$, beside $A^{-2}B_i\in\cL(U,Y)$, $i=1,2$.  
\end{proof}


\subsubsection{The optimal cost operator} \label{s:riccati-operator}
We  note that the Riccati Operator defined via optimal trajectory  coincides with 
\begin{equation}\label{e:riccati-operator-1}
P(t) \alpha = \int_t^T e^{A^*(\tau-t)}R^*R \Phi(\tau,t)\alpha\,d\tau\,,
\qquad 0\le t\le T\,, \; \alpha\in [\cD({A^*}^2)]'\,,
\end{equation} 
where $\Phi(\tau,t)$ is defined in \eqref{e:transition}.
It is readily seen that, combining
$\Phi(\tau, t)\alpha=\hat{y}(\tau,t;\alpha)-B_1 \hat{g}(\tau,t;\alpha)$ with $RB_1=0$, 
\eqref{e:riccati-operator-1} is actually equivalently rewritten as follows
\begin{equation*} 
P(t) \alpha = \int_t^T e^{A^*(\tau-t)}R^*R \hat{y}(\tau,t;\alpha)\,d\tau\,,
\qquad 0\le t\le T\,, \; \alpha\in [\cD({A^*}^2)]' 
\end{equation*} 
which confirms the equivalent relation \eqref{e:riccati-operator-2}.

\begin{proposition} \label{p:Riccati-operator}
The optimal cost operator $P(t)$ defined by \eqref{e:riccati-operator-1} (equivalently,
by \eqref{e:riccati-operator-2}) satisfies the following (enhanced) regularity properties:

\begin{enumerate}
\item
{\rm (\bf Space regularity)}
For any given $t\in [0,T]$, one has 
\begin{equation}\label{e:riccati-spaceregularity}
{A^*}^2 P(t) A^2 \in \cL(Y)\,;
\end{equation}
equivalently,
\begin{equation}
P(t)\in \cL([\cD({A^*}^{\gamma_1})]',\cD({A^*}^{\gamma_2})) 
\qquad \forall \gamma_1, \gamma_2\le 2\,.
\end{equation}
%

As a consequence, $B_i^*P(\cdot)A^2\in \cL(Y,U)$, $i=1,2$ and
the gain operator $B^*P(t)\equiv B_0^*P(t)+B_1^*A^*P(t)$ satisfies $B^*P(t) A^2 \in \cL(Y,U)$;
namely,
\begin{equation}\label{gain-spaceregularity}
B_i^*P(t)\in \cL([\cD({A^*}^2)]',U))\,; \qquad i =0,1\,.
\end{equation}
%
\item
{\bf (Time regularity)} As for the regularity in time of the optimal cost operator -- then,
of the value function -- one has
\begin{equation}\label{e:riccati-timeregularity}
P(\cdot) \;\textrm{continuous} \colon 
[\cD({A^*}^2)]' \longrightarrow C(0,T;\cD({A^*}^2))\,.
\end{equation}

\end{enumerate}
\end{proposition}

\begin{proof}
1. Let $\alpha\in [\cD({A^*}^2)]'$ be given. 
We write down and compute, with $0\le t\le T$, 
\begin{equation*}
\begin{split}
(-A^*)^2P(t)\alpha &= (-A^*)^2\int_t^T e^{A^*(\tau-t)}R^*R \Phi(\tau,t)
\alpha\,d\tau 
\\[1mm]
& =  \int_t^T e^{A^*(\tau-t)}[(-A^*)^2R^*]\,R \Phi(\tau,t)\alpha\,d\tau 
\end{split}
\end{equation*}
where the application of the operator $(-A^*)^2$ commutes with the integration
in time on the extrapolation space.
 
Then, the conclusion in \eqref{e:riccati-spaceregularity} follows recalling that
the function $R \Phi(\cdot,t)\alpha$ takes values in $Y$ (cf.~\eqref{e:transition-item2}), whilst $(-A^*)^2R^*$ is a bounded operator on $Y$.
\\
As for gain operator, on the basis of \eqref{e:riccati-spaceregularity}, we next obtain 
\begin{equation*}
B_i^*P(\cdot)A^2=[B_i^*(-A^*)^{-\gamma_0}]\,(-A^*)^{\gamma_0}P(\cdot)A^2 \in \cL(Y,U)\,,
\qquad i=1,2\,,
\end{equation*}
owing to $B_i^*(-A^*)^{-\gamma_0}\in \cL(Y,U)$, $ \gamma_0  =1 $  thereby confirming the exceptional 
boundedness and smoothing effect of the gain operator in \eqref{gain-spaceregularity}.

\smallskip
\noindent
2. Finally, the regularity in time of \eqref{e:riccati-timeregularity}
follows combining the continuity in time of the function $R \Phi(\cdot,t)\alpha$ 
(see, once again, \eqref{e:transition-item2}) with more standard semigroup properties;
see the proof in \cite[p.~697]{redbook}.
\end{proof}

\subsubsection{The Riccati equation}
In this section we shall provide several key relations which lead to a characterization of the Riccati operator via Differential Riccati equation. 
One of the fundamental properties is time evolution (of the evolution operator) with respect to the initial time, that is the second argument. 
In the case of semigroups both evolutions are the same. 
However, in the case of time dependent evolutions -- as in the present case -- proving differentiability with respect to the initial time is challenging.
The challenge is due to compromised regularity and the intrinsic lack of invariance. 

\begin{lemma}[Differentiability of the evolution with respect to initial time]\label{Right}

The evolution operator $\Phi(\tau,t)$ defined in \eqref{e:transition} satisfies
\begin{equation*}
\frac{d}{dt} \big (R\Phi(\tau,t)\alpha\big)=- R\Phi(\tau,t)\big[A-BB^*P(t)\big]\alpha
\qquad \forall \alpha \in [\cD({A^*})]'\,, \quad \textrm{a.e. in $t$,}
\end{equation*}
where $B$ denotes $B_0 + A B_1$.
\end{lemma}

\begin{proof}
We will sketch the major steps of the proof.

\noindent
1. We have seen that $R\Phi(t,s)$ may be defined on the extrapolation space 
$[\cD({A^*}^2)]'$.
In particular, it makes sense $R\Phi(t,s)Bu$ and it holds
\begin{equation*}
\sup_{0\le t\le T}\|R\Phi(\cdot,t)Bu\|_{L^1(t,T;Y)}\le c_T\|u\|_U\,.
\end{equation*}
To justify the above assertion: we recall that 
\begin{equation*}
R\Phi(\cdot,t)\alpha=Re^{A(\tau-t)}\alpha+R\{L_t \hat{g}(\cdot,t;\alpha)\}(\tau) 
\end{equation*}
which combined with \eqref{e:optimal-control-from-optimality} gives 
\begin{equation}\label{e:rphi}
R\Phi(\tau,t)\alpha = \big\{ \big[I+RL_tL_t^*R^*\big]^{-1} \, Re^{A(\tau-t)}\alpha\big\}(\tau)\,,
\quad \alpha \in [\cD({A^*}^2)]'
\end{equation}
Insertion of $Bu \in [D(A^{2*})]'$ in place of $\alpha$ brings about the estimate 
\begin{equation*}
\sup_{0\le t\le T}\|R\Phi(\cdot,t)Bu\|_{L^1(t,T;Y)}\le \dots \le 
\|Re^{A(\tau-t)}\alpha\|_{L^?(t,T;Y)}
\le c_T\|u\|_U\,.
\end{equation*}

\smallskip
\noindent
2. A major step is to show existence (as well as to pinpoint the regularity) of the derivative
of $R\Phi(\tau,t)\alpha$ with respect to $t$, with $\alpha$ belonging to the largest possible space. 
The arguments here owe to \cite[Vol.~II, Lemma~8.3.4.2]{redbook}. 
\\
Rewrite
\begin{equation}\label{e:implicit-eq-for-rphi}
R\Phi(\tau,t)\alpha+ \big\{RL_tL_t^*R^*\, R \Phi(\cdot,t)\alpha\big\}(\tau)
= Re^{A(\tau-t)}\alpha
\end{equation}
and notice that if $\alpha\in [\cD(A^*)]'$
(please note that here it is {\bf not} $\alpha\in [\cD({A^*}^2)]'$), 
then 
\begin{equation*}
Re^{A(\tau-t)} x= RA^2\,A^{-1}e^{A(\tau-t)}A^{-1}x\,,
\end{equation*}
which gives 
\begin{equation*}
\frac{d}{dt} Re^{A(\tau-t)} x= -[RA^2]e^{A(\tau-t)} \underbrace{A^{-1}x}_{\in H}\,.
\end{equation*}
Rewrite next \eqref{e:implicit-eq-for-rphi} explicitly:
\begin{equation*}
\begin{split}
& R\Phi(\tau,t)\alpha + R\int_t^\tau e^{A(\tau-\sigma)}B\int_\sigma^T B^*e^{A^*(r-\sigma)}
R^*R \Phi(r,t)\alpha\, dr\, d\sigma=
\\[1mm]
& \myspace = Re^{A(\tau-t)}\alpha
\end{split}
\end{equation*}
which implies
\begin{equation*}
\begin{split}
& \frac{d}{dt} \big(R\Phi(\tau,t)\alpha\big)- Re^{A(\tau-t)}B\int_t^T B^*e^{A^*(r-t)}
R^*R \Phi(r,t)\alpha\, dr \,+
\\[1mm]
& \qquad +\,R\int_t^\tau e^{A(\tau-\sigma)}B\int_\sigma^T B^*e^{A^*(r-\sigma)}
R^*R \frac{d}{dt} \big(R\Phi(\tau,t)\alpha\big)\, dr\, d\sigma
\\[1mm]
& \myspace =- Re^{A(\tau-t)}A\alpha\,.
\end{split}
\end{equation*}
The above implicit equation is rewritten as
\begin{equation*}
\big[I+RL_tL_t^*R^*\big]\frac{d}{dt}\big(R\Phi(\cdot,t)\alpha\big) =
-\underbrace{Re^{A(\tau-t)}A\alpha}_{T_1(\tau,t)}+ \underbrace{Re^{A(\tau-t)}BB^*P(t)\alpha}_{T_2(\tau,t)}
\end{equation*}
which makes sense at least in $H^{-1}(0,T;Y)$.

Then, noting that 
\begin{equation*}
T_1(\cdot,t)\in C([t,T];Y)\,, \qquad T_2(\cdot,t)\in L^\infty(t,T;Y)
\end{equation*}
we get 
\begin{equation*}
\frac{d}{dt} \big(R\Phi(\tau,t)\alpha\big)=\big[I+RL_tL_t^*R^*\big]^{-1}
\Big\{- Re^{A(\tau-t)}A\alpha+Re^{A(\tau-t)}BB^*P(t)\alpha\Big\}\in L^2(t,T;HY\,.
\end{equation*}
Recalling \eqref{e:rphi} we finally obtain
\begin{equation*}
\frac{d}{dt} \big(R\Phi(\tau,t)\alpha\big)=- R\Phi(\tau,t)A\alpha+R\Phi(\tau,t)BB^*P(t)\alpha
\end{equation*}
(cf.~\cite[Vol.~II, \S~8.3.4, p.~701]{redbook}), thereby providing with 
\begin{equation*}
\begin{split}
& \frac{d}{dt} (\big(R\Phi(\tau,t)x\big),y)_Y=
\\[1mm]
& \qquad
=- (R\Phi(\tau,t)\,[A-(B_0 + A B_1 )\,(B_0^* + B_1^* A^*)P(t)]x, y)_Y\,,
\quad x\in [\cD(A^*)]', y\in Y\,.
\end{split}
\end{equation*} 
\end{proof}


\begin{lemma}[\bf First Feedback Synthesis] \label{l:F}
The optimal control $\hat{g}$ admits the representation
\begin{equation*}
\hat{g}(\tau,t;\alpha) =- [B_0^* + B_1^* A^*] P(\tau)\Phi(\tau,t)\alpha
\qquad \forall \alpha \in [\cD({A^*}^2)]'\,.
\end{equation*}

\end{lemma}

\begin{proof}
From the optimality conditions we know that 
\begin{equation*}
\hat{g}(\tau,t;\alpha) =-\{L_t^*R^*R\hat{y}(\cdot,t;\alpha)\}(\tau)\,.
\end{equation*}
Because $RB_1 =0$, and exploiting the evolution property enjoyed by $\Phi$, it follows
\begin{equation*}
\hat{g}(\tau,t;\alpha) =-L_t^*R^*R \Phi(\cdot,t) \alpha\,.
\end{equation*}
%
Observing that for any $\alpha \in [\cD({A^*})]'$ one has $R\Phi(t,s) \alpha \in Y$ and 
$L_t^* R^* \colon L^1(Y)  \rightarrow C(U)$, makes the above composition of operators 
meaningful -- as acting on appropriate domains. 
This concludes the optimal synthesis as stated in the Lemma. 
\end{proof}

\begin{lemma}[\bf Riccati Equation]\label{RIC}
For all $x, y \in \cD(A)$ the Riccati operator $P(\cdot)$ satisfies  
\begin{equation*}
\begin{split}
& \big(\frac{d}{dt} P(t)x,y\big)_Y= -(R^*Rx,y)_H-(A^* P(t)x,y)_Y -
\\[1mm]
& \myspace - (P(t)Ax,y)_Y- ([B_0^*+B_1^*A^*]P(t)x,[B_0^*+B_1^*A^*]P(t)y)_Y\,,
\end{split}
\end{equation*}
with 
\begin{equation*}
\begin{cases}
A^*P_t(t) A \in \cL(Y)\,, 
\\[1mm]
\textrm{$A^*P_t(t) A$ continuous $\colon Y\longrightarrow L^\infty(0,T;Y)$.}
\end{cases}
\end{equation*}

\end{lemma}

\begin{proof}
In order to derive the Riccati equation, we follow the so called direct approach (cf.~\cite{redbook}).
Differentiation (in a weak sense) of the Riccati operator requires the characterization of the left derivative  (with respect to the initial time) of the evolution.
However, in the present case, Proposition~\ref{p:Phi} provides the needed regularity for the evolution when acted upon by the observation.   
This allows to obtain the critical representation for the right evolutionary derivative which is  given by Lemma~\ref{Right}. 
The said representation, when combined with the ``first feedback synthesis'' in  Lemma~\ref{l:F}
gives the final conclusion.

Calculations are justified by the already proved regularity of the quantities involved. 
In particular, the compromised regularity of the derivative of the evolution (which requires 
$\alpha \in [\cD(A^*)]'$, is sufficient to obtain the final conclusion. 
\end{proof}

We note that the feedback synthesis given in Lemma \ref{l:F} is in terms of the evolution operator $\Phi(t,s)$. 
What is needed, instead, is the feedback synthesis in terms of the actual trajectory $\hat{y}$.
This is attained below. 

\begin{lemma}[\bf Feedback Synthesis] \label{l:feed} 
For any $\alpha \in [\cD({A^{2*}}]'$, the following feedback representation of the
optimal control $\hat{g}(t;\alpha)$ holds true:   
\begin{equation*}
\hat{g}(t;\alpha) =- \big[I - [B_0^* + B_1^* A^*] P(t) B_1\big]^{-1} 
[B_0^* + B_1^* A^*]P(t)\hat{y}(t,\alpha)\,; 
\end{equation*}
the formula provides an ``on line'' optimal control $\hat{g}(\cdot, \alpha ) \in L_2(U)$
for the $\alpha$-parametrized problem. 
\end{lemma}

\begin{proof}
For the feedback synthesis of the optimal control it remains to discuss the invertibility of the operator 
\begin{equation*}
I-[{B_0}^*+{B_1}^*A^*]P(t)B_1\,.
\end{equation*}

\begin{proposition}\label{p:feed}
The operator $I-[{B_0}^*+{B_1}^*A^*]P(t)B_1$ is boundedly invertible on $U$ for each 
$t \in [0, T]$. 
\end{proposition}

\begin{proof}
{\sl Step 1.} 
We shall first prove the injectivity of the operator 
$I-[{B_0}^*+{B_1}^*A^*]P(t)B_1$ for $t =0$. 
Then, the dynamic programming argument extends the argument to all $t\in [0,T]$.

By contradiction, let $v \in U$ be such that $v \ne 0$, and 
\begin{equation}\label{1}
v=[{B_0}^*+{B_1}^*A^*]P(t)B_1v\,.
\end{equation}
Consider then the optimal control problem with $y_0=0$, and $\alpha =-B_1v$.
The (implicit) optimal synthesis gives 
\begin{equation}\label{2}
\hat{g}_{\alpha}(0) = - [B_0^* + B_1^* A^*] P(0) 
\big(\hat{ y}_{\alpha}(0) - B_1 \hat{g}_{\alpha} (0)\big)\,.
\end{equation}
But from the continuity of optimal control, we also have 
$\hat{y}_{\alpha}(0) = \alpha + B_1 \hat{g}_{\alpha}(0)$. 
This, combined with \eqref{2} give
\begin{equation}\label{3}
\hat{g}_{\alpha}(0) =
- [B_0^* + B_1^* A^*] P(0) [\hat{ y}_{\alpha}(0) - B_1 \hat{g}_{\alpha} (0) ]
= [B_0^* + B_1^* A^*] P(0) B_1v\,. 
\end{equation}
From the contradiction argument \eqref{1} it follows that $g^0_{\alpha}(0) =v$.
On the other hand, the optimal control problem with $y_0 =0$ produces only one solution 
which is equal identically to zero. 
Therefore, the optimal control $g^0$ should be zero as well. 
This contradicts the fact that $v \ne 0$. 

The same argument applied to the dynamics originating at the time $t$ yields
injectivity of $I - [B_0^* + B_1^* P(t) ]B_1$ on $U$, for any $t \in [0,T]$. 

\smallskip
\noindent
{\sl Step 2}. 
Compactness of the operator $[B_0^* + B_1^* P(t)]B_1$. 
This follows from regularity properties of $P(t)$ which asserts that 
$P(t)\colon \cD({A^*}^2)]' \rightarrow \cD({A^*}^2)$ is bounded. 
However, the injection $B_1 \colon U \rightarrow \cD({A^*}^2)$ is compact.
The latter follows from the fact 
$A^{-1} B_1g = [ b c^{-2} \cA^{-1} ( \cA + I ) N_0 g, 0, 0]$ and elliptic theory giving 
$N_0 : L_2(\Gamma_0 ) \rightarrow H^1(\Omega) $  is compact. 

Thus, the final conclusion follows from spectral theory of compact operator. 
\end{proof}

Now, the conclusion in Lemma \ref{l:feed} follows from the Proposition \ref{p:feed} and the representation  in Lemma \ref{l:F} supported by definition of evolution operator $\Phi$. 
\end{proof}
Completion of the proof of Proposition \ref{T}: combine the results stated in Proposition \ref{p:Riccati-operator}, Lemma \ref{RIC} and Lemma \ref{l:feed}. 

Completion of the proof of Theorem \ref{T0}: setting $\alpha = y_0 - B_1 g_0$  provides the conclusions stated in Theorem \ref{T0}. 
 
\subsection{Proof of Theorem \ref{T:1}}
It remains to be shown that $\hat{g}(0)$ coincides with the parameter $g_0$.
This is done below. 

Let $y_0 \in [\cD({A^2}^*)]'$ and $g_0 \in U$ be given.
With $\alpha= y_0 - B_1g_0$, we know from from Part 1 of Theorem \ref{T0} that the optimal control $g^0$ belongs $C([0,T];U)$. 
Therefore, in order to comply with the original model one is asking for the following selection of the parameter $g_0$: $g_0 = g^0(0)$. 
This amounts to 
\begin{equation*}
g^0_\alpha(t=0)=g_0\,, \qquad  \alpha = y_0 - B_1g_0\,.
\end{equation*}
The above implicit relation is always uniquely solvable for some  $g_0 \in U$.
In fact, the matching condition amounts to solving 
$g_0 = F \alpha = F (y_0 - B_1 g_0 )$, that is $(I - F B_1) g_0 = Fy_0$,
where $F\equiv [B_0^* + B_1^* A^*] P(0)$.

With the key properties $F \in \cL([\cD(A^*)]',U)$ and $(I - FB_1)^{-1}\in \cL(U)$.
However, we recognize that $I - F B_1$ coincides with the operator $G(0)$, for which the requisite boundeness and invertibility have been shown in Proposition \ref{p:feed}. 

Thus we  obtain 
\begin{corollary}
Let $y_0 \in \cD({A^2}^*)]'$ be given. 
Consider Problem $\mathcal{P}_{\alpha}$ with $\alpha = y_0 -B_1g_0$ and $g_0\in U$
given by
\begin{equation}\label{g00}
g_0 =  (I - FB_1 )^{-1} F y_0\,. 
 \end{equation}
Then, there exists a unique optimal control $g^0\in  C([0,T]; U)$ and a 
corresponding trajectory \eqref{e:eq-for-U}, with $y^0(0) = y_0$, 
such that the results of Proposition \ref{T} hold with $\alpha= y_0 - B_1 g_0$ and 
$g_0$ given by \eqref{g00}.
\end{corollary}
 
In other words, by solving the parametrized optimal control problem with a given 
$\alpha = y_0 - B_1 g_0$ and a parameter $g_0 \in U $ we solve a family of parametrized optimal control problems, which always has a unique solution. 
The original  dynamics is included in this  family. 
By selecting $g_0\in U$ according to the matching condition, we make a selection of a problem whose dynamics coincides with the original one.
However, the above does not imply that the constructed optimal control for the parametrized  control problem is also optimal for the original problem -- when considered within the $L_2(U)$ framework for optimal controls.
In fact, the latter may not have an optimal solution at all when $y_0 \in \cR(B_1)$, as shown in  Theorem \ref{l:neg}; see also \cite{LPT}.
Thus, the constructed control is suboptimal, yet it corresponds to the original dynamics. 
However, if the original problem does have an $L_2(U)$ optimal control, then such control coincides with a parametrized control where $g_0$ is selected according to the matching condition.
 
\subsection{Proof of Theorem \ref{T:2}}
Theorem \ref{T:2} follows from Theorem \ref{T:1} by using a rather standard argument in  calculus of variations. 
To wit: we recall from Proposition \ref{T} that the optimal value for the parametrized problem equals  
\begin{equation*} 
J(\hat{g},\hat{y}_{g_0}) = (P(0) \alpha, \alpha )_Y  
= (P(0) (y_0 - B_1 g_0), y_0 - B_1 g_0 )_Y\,. 
\end{equation*}
On the strength of positivity and selfadjointness of $P(0)$ we can write the above as 
\begin{equation*} 
J(\hat{g},\hat{y}_{g_0}) = || P^{1/2} (0) (y_0 - B_1 g_0 ) ||^2_Y\,.
\end{equation*}
Appealing to the regularity properies of $P(0)$ listed in Theorem \ref{T:1} we obtain that 
$J(g_0) \equiv J(\hat{g},\hat{y}_{g_0})$ is weakly lower semicontinuous on $U$. 
Indeed, the latter follows from 
\begin{eqnarray}\label{opt}
J(\hat{g},\hat{y}_{g_0})= (P(0) (y_0 - B_1 g_0), y_0 - B_1 g_0 )_Y = (P(0) y_0, y_0 )_Y 
\notag  \\ 
- 2 (P(0) y_0, B_1 g_0 )_Y +  (P(0) B_1 g_0, B_1 g_0 )_Y\,,
\end{eqnarray}
where $A^{-1} B_1 \colon U \rightarrow Y$ is compact and $A^*P(0) A \colon Y \rightarrow Y$
is bounded. 
This gives  compactness of the map  $g \rightarrow P^{1/2}(0) B_1 g$  from $U$ to $Y$,
adressing the convergence of the last quadratic term in \eqref{opt}.
 
As for the first term, we simply recall   Proposition \ref{p:Riccati-operator}  which states  $ {A^*}^2P(0) A^2 :Y \rightarrow Y $ is also  bounded.
Strong continuity of the  second term (linear in $g_0$ )  follows now from $ A^{-1} B_1 \in L(Y) $ and $A^* P(0) A^2 \in L(Y) $. 
Thus the  regularity of the Riccati operator $ P(0)$ along with  $ A^{-1} B_1 \in L(Y) $  implies  weal lower-semicontinuity of the functional.   Since $U_0$ is weakly compact, we obtain a minimizing sequence $ g_n \in U_0$ such that $J(g_n) \rightarrow d  = \inf_ {g_0\in {U_0}} J(g_0) $ 
and $g_n \rightarrow g^* \in U_0 $ weakly in $U$. 
 Weak lower semicontinuity  of  $ J(g_0)$  gives an existence of a minimizer. The characterization of the minimizer follows now  from a standard argument in  calculus of variations,
after taking into consideration  the  representation of the functional via Riccati operator.  This leads to the final conclusion  stated in  Theorem \ref{T:2}.
 
\ifdefined\xxx
1. Notice that unlike the more general framework of Lasiecka-Lukes-Pandolfi and 
Lasiecka-Pandolfi-Triggiani, in the present case there is no need of an analysis 
of a non-standard LQ-problem, along with the invocation of a Dissipation Inequality
satisfied by the value function. 
Indeed, in view of $RB_1\equiv 0$, the dynamics $y(t) -B_1g(t)$ is such that 
\begin{equation}
\|R(y(t) -B_1g(t))\|^2\equiv \|Ry(t)\|_H^2 
\end{equation}
and the cost functional is the same for both dynamics.

\smallskip
\noindent
2. I have analyzed the case $d>0$ and $d_0=0$: we have once more $RA^2$ bounded and $RB_1=0$;
it must be computed the value of $\lambda\ne 0$ for wich $(\lambda-A)^{-1}B_i$ is bounded. 
We will return on this after your first feedback.
\end{remarks}
\fi


\section*{Acknowledgements}
The authors are grateful to Barbara Kaltenbacher, whose work has provided motivation for studying 
control problems associated with the SMGT acoustic model.  
Inspiring and illuminating mathematical conversations of both authors with Barbara are gratefully
acknowledged. 

\smallskip
\noindent
The research of F.B. was partially supported by the Universit\`a degli Studi di Firenze under the Project 
{\em Analisi e controllo di sistemi di Equazioni a Derivate Parziali di evoluzione}, and by the GDRE (Groupement de Recherche Europ\'een) ConEDP (Control of PDEs). 
F.B. is a member of the Gruppo Nazionale per l'Analisi Matematica, la Probabilit\`a e le loro Applicazioni (GNAMPA) of the Istituto Nazionale di Alta Matematica (INdAM), whose occasional support is acknowledged. 
\\
The research of I.L. was  partially supported by the NSF Grant DMS-1713506.


\end{document}